\documentclass[reqno,12pt]{amsart}
\usepackage{willie_style}

\usepackage{xr-hyper}
\title{Quantum computing with anyons: an $F$-matrix and braid calculator}
\author[W.~Aboumrad]{Willie Aboumrad}
\address[W.~Aboumrad]{The Institute for Computational and Mathematical Engineering (ICME) at Stanford University}
\email{willieab@stanford.edu}
\urladdr{https://web.stanford.edu/~willieab}

\date{}

\begin{document}
\maketitle

\begin{abstract}
	We introduce a pentagon equation solver, available as part of {\sc SageMath}, and use it to construct braid group representations associated to certain anyon systems. We recall the category-theoretic framework for topological quantum computation to explain how these representations describe the sets of logical gates available to an anyonic quantum computer for information processing. In doing so, we avoid venturing deep into topological or conformal quantum field theory. Instead, we present anyons abstractly as sets of labels together with a collection of data satisfying a number of axioms, including the pentagon and hexagon equations, and explain how these data characterize ribbon fusion categories (RFCs). In the language of RFCs, our solver can produce $F$-matrices for anyon systems corresponding to multiplicity-free fusion rings arising in connection with the representation theory of quantum groups associated to simple Lie algebras with deformation parameter a root of unity. 
\end{abstract}

\section{Introduction}
\label{anyonic qc ch}
\longer{In this \chorpaper\ we introduce the anyonic computing model and discuss the capabilities of the \texttt{FusionRing} class available in {\sc SageMath} for simulating anyonic quantum computers.}

Two decades ago Freedman and Kitaev proposed a model for topological quantum computation (TQC) featuring a major advantage over the mainstream qubit model. Their model encodes information in global properties of topological phases of matter that are invariant under local deformations. The encoded information is intrinsically protected at the physical hardware level so the TQC model boasts a much higher error threshold. The series of papers \cite{fkw,flw,fklw,kitaev_2003} introduced and advanced the TQC model; \Cref{rfc framework for tqc} recalls the basics from an algebraic standpoint.

Topological quantum computers are built on \textit{anyons}, which are point-like quasiparticles believed to exist in certain $2$-dimensional materials. They can be understood as primary fields in Wess-Zumino-Witten conformal field theories \cite[Chapters~15-17]{francesco_mathieu_senechal_2011} and they can also be described in terms of Chern-Simons theory \cite{iengo_1992}. In experimental physics, anyons are expected to appear in fractional quantum Hall liquids and other superconductors; \cite{nayak_simon_stern_freedman} surveys these and other approaches.

Unlike bosons or fermions, \textit{any}ons support exotic particle statistics; in some sense, \textit{any}thing can happen because anyons are constrained to two spatial dimensions. In a $(2 + 1)$-dimensional spacetime, not all loops are topologically equivalent: the fundamental group of the configuration space over $\mathbb{R}^2$ is the Artin braid group. This means anyonic statistics are described by (possibly high-dimensional) representations of the braid group. 

Understanding these representations is tantamount to understanding the computational power of an anyonic quantum computer because anyon exchanges are the physical process driving information processing in the TQC model. This \chorpaper\ explains how to obtain these representations and how to construct their associated matrices explicitly using {\sc SageMath}. In particular, \Cref{numerical braids} features explicit braid computations.

We begin by presenting anyon systems abstractly in \Cref{anyon system} as sets of labels together with complex numbers satisfying a number of axioms. The label set is equipped with commutative associative addition and multiplication operations with integral structure coefficients making it into a gadget known as a \textit{fusion ring} or \textit{Verlinde algebra}. Fusion rings are well-known and well-studied, e.g., in \cite{fuchs,feingold_2004}, and they arise in various areas of modern physics and mathematics, where they describe the possible couplings amongst three objects of some given class. For example, fusion rings describe: 
\begin{itemize}
	\item The decomposition into irreducibles of tensor products of finite-dimensional representations of semisimple complex Lie algebras, finite groups, or associative bialgebras;
	\item Truncated tensor products of unitary representations of quantum groups with deformation parameter at a root of unity 	\cite{alvarez_gomez_sierra_1990,chari_pressley_1994};
	\item The composition of superselection sectors in the $C^*$-algebraic approach to relativistic quantum field theory \cite{frs_89};
	\item The coupling of primary fields of $\mathcal{W}$-algebras in two-dimensional conformal field theory \cite{dijkgraaf_verlinde_verlinde_1988};
	\item Operator products in topological quantum field theory \cite[Chapter~6]{francesco_mathieu_senechal_2011}.
\end{itemize}
With the exception of the last one, the examples above appeared in \cite{fuchs}. 

In addition to a label set that defines a fusion ring, an anyon system consists of three sets of complex numbers: a $6j$-system or $F$-matrix, a $3j$-system or $R$-matrix, and a pivotal structure. 

We stress this data-driven approach to anyons for two reasons: on one hand, it allows us to readily develop a category-theoretic algebraic framework for TQC without going deep into conformal or topological quantum field theory;  on the other, it allows us to leverage the \texttt{FusionRing} class in {\sc SageMath} to perform explicit computations in many interesting cases. The \texttt{FusionRing} class models the data defining anyon systems that arise in connection with the finite-dimensional representation theory of a quantum group associated to a simple complex Lie algebra with deformation parameter a root of unity. The \texttt{FusionRing} class can compute $R$-, and $S$-matrices, twists, quantum dimensions, and invariants such as the Virasoro central charge and global topological order, amongst others; detailed documentation may be found at \url{https://doc.sagemath.org/html/en/reference/combinat/sage/combinat/root_system/fusion_ring.html}. Moreover, when the anyon system is multiplicity-free, the \texttt{FusionRing} class can obtain an associated $F$-matrix. The class was implemented by the author together with Daniel Bump and Travis Scrimshaw. 

With the notion of anyon systems in hand, we argue in \Cref{abcs of tqc} that certain ribbon fusion categories (RFCs) capture the dynamics of anyon systems.  \ddichotomy{\Cref{prelude} covers}{We have supplied a companion article \cite{willie_rfc} covering} the necessary background on category theory. While this material is well-known and may be found in various sources, e.g., \cite{baki,etingof_gelaki_nikshych_ostrik_2017}, our companion \chorpaper\ aims to be concise and presents only the necessary notions. 

\Cref{braiding via frt mats} then computes the desired braid representations by leveraging the categorical framework to obtain formulas expressing braid generators in terms of $F$- and $R$-matrices. \Cref{general braiding formula} collects these formulas.

In practice, however, we do not have direct access to all the data specifying an anyon system; instead, we typically deal with the corresponding RFC. In particular, in this \chorpaper\ we focus on RFCs arising in connection with the representation theory of quantum groups. In this case, we may easily obtain the $R$-matrix and the pivotal structure of the corresponding anyon system using quantum group theory; for details, refer to the documentation of the \texttt{FusionRing.twist} and \texttt{FusionRing.r\_matrix} methods. Obtaining an $F$-matrix, however, requires some work: we must solve the Pentagon Equations \eqref{pent eqn}. 

These polynomial relations enforce coherence for the monoidal structure of the RFC. Ocneanu rigidity guarantees there are only finitely many solutions, so we can always solve the system using Groebner basis methods, \textit{in principle} \cite{buchberger_1976}. 

In practice, however, these methods are prohibitively slow. The algorithms scale at least exponentially, in both space and time, as a function of the number of variables, equations, degrees of the equations, and the number of solutions to the equations; in an anyon system of $N$ labels, the number of variables in the pentagon system scales like $N^6$ while the number of equations scales like $N^9$. Thus ``solving the pentagons turns out to be a difficult task (even with the help of computers)'' \cite[Section~2.4]{trebst_troyer_wang_ludwig_2008}.

In some sense completing this task is more of an art than a science and engineering takes a leading role. For this we implemented a solver in {\sc SageMath} capable of obtaining an $F$-matrix associated to a given multiplicity-free \texttt{FusionRing} object. In fact there are two solvers available, but here we focus on the orthogonal solver. The solver is managed by the \texttt{FMatrix} class, and its open-source implementation is currently available in the Trac server at \url{https://trac.sagemath.org/ticket/30423}. The code may be used upon building the {\sc SageMath} development branch, as explained here: \url{https://doc.sagemath.org/html/en/developer/walk_through.html}. The \texttt{FMatrix} class is currently set to merge into the stable {\sc SageMath} $9.8$ release.

Although our solver ultimately relies on Groebner basis computations, it implements a variety of techniques to ensure the calculation remains manageable in many interesting cases. In particular, the solver is at least capable of producing an $F$-matrix associated with any multiplicity-free \texttt{FusionRing} with a dozen generators using the computational resources of a laptop. Since the solver relies on parallel computations, it would likely benefit from access to more processing cores.

An important feature is that we exploit the Hexagon Equations \eqref{hex eqn}, which enforce consistency between the associator and the braiding on an RFC. The hexagons are much more \textit{localized} than the pentagons: the equations graph they induce consists of many relatively small connected components that may be processed independently and in parallel. In addition, the hexagons help determine a field containing the $F$-matrix as they feature $3j$-symbols in the form of cyclotomic coefficients. \Cref{solver details} discusses this and other features of our implementation in further detail, including our custom {\sc Cython} arithmetic engine for sparse polynomials, the concurrent programming involved, and the shared data structures needed to support it.

We note that some attempts at solving the pentagon system have appeared in the literature and some solutions exist for various special cases. For example, see \cite{hagge,rowell_stong_wang_2009,trebst_troyer_wang_ludwig_2008,bonderson,cui_wang_2015,f_symbols_h3}. However, to the best of our knowledge, including the hexagons is a novelty of our solver. In addition, our solver implementation seems to be the first open-source one available, and the first one that can handle a variety of fusion rings in a unified framework.

\subsection*{Acknowledgements} 
The author is greatly indebted to Daniel Bump and Travis Scrimshaw for their collaboration, hard work, patience, and feedback in developing the \texttt{FusionRing} code. The author would also like to thank Eric Rowell for useful discussions and guidance, especially regarding a formula for computing certain $R$-symbols. In addition, the author thanks David Roe and Julian Ruth, amongst others, for useful discussions about {\sc SageMath} implementation details during the {\sc SageMath} Days Conference in May 2021.

\section{ABCs of TQC: Anyons, Braids, and Categories for Topological Quantum Computing}
\label{abcs of tqc}
In this section we define anyons systems and explain how to realize anyonic quantum computation using braiding operators acting on fusion spaces. We show that anyons can be interpreted as simple objects in certain RFCs, in the sense of \crossrefrfc{rfc defn}. Technically, anyon systems define unitary modular tensor categories (UMTCs), which are RFCs equipped with a collection of Hermitian forms on hom-sets that is compatible with the associator and whose braiding satisfies a non-degeneracy condition. We attempt to avoid venturing too far into the weeds of category theory and instead focus on the structure that is necessary for the anyonic quantum computation model. In addition, we only consider \textit{multiplicity-free} anyon systems.

Without further ado, we begin with a definition.

\begin{dfn}\label[defn]{anyon system}
	An \textit{anyon system} is a finite set $\ccat$ of labels together with the following data. 
	\begin{enumerate}[(i)]
		\item The set $\ccat$ is equipped with the structure of a \textit{fusion ring}, which means there is a commutative associative \textit{fusion rule} $\times$ and integral structure constants $N^{ab}_c \geq 0$ such that for every $(a, b) \in \ccat^2$, 
		\begin{equation}\label{fusion rule prototype}
			a \times b = \sum_{c \in \ccat} N^{ab}_c c.
		\end{equation}
		In this work, we only consider \textit{multiplicity-free} fusion rules: every $N^{ab}_c \in \{0, 1\}$. 
		
		A triple $(a, b, c) \in \ccat^3$ is \textit{admissible} if $N^{ab}_c \neq 0$. We will sometimes write $c \in \{a \times b\}$ to mean $(a, b, c)$ is admissible. A sextuple $(a, b, c, d, e, f)$ is \textit{admissible} if $e \in \{a \times b\}$, $d \in \{e \times c\}$, and $f \in \{b \times c\}$.
		
		\item There is a unique \textit{trivial} or \textit{vacuum} label $\mathbf{1} \in \ccat$ such that for every $a \in \ccat$,
		$$a \times \mathbf{1} = a = \mathbf{1} \times a.$$
		
		\item The set $\ccat$ is equipped with an involutive (ring) automorphism $* \colon \ccat \to \ccat$, called a \textit{conjugation} or \textit{duality operator}, such that for every $a \in \ccat$, 	
		$$
			\mathbf{1} \in \{a \times b\} 
			\quad \text{if and only if} \quad 
			b = a^*.
		$$
		
		\item There is an \textit{$F$-matrix} or \textit{$6j$-system} $F\colon \ccat^6 \to \mathbb{C}$ satisfying the following conditions. Write $F^{abc}_{d; \, ef} = \big[F^{abc}_d \big]_{e f}$ for $F(a, b, c, d, e, f)$ and let $F^{abc}_d$ denote the matrix with $(e, f)$-entry $F^{abc}_{d; \, ef}$. 
		\begin{enumerate}
			\item \textit{(Admissibility.)} If $(a, b, c, d, e, f)$ is not admissible, then $F^{abc}_{d; \, ef} = 0$. 
			\item \textit{(Pentagon Axiom.)} For every nonuple $(a, b, c, d, e, f, g, k, \ell) \in \ccat^9$, 
			\begin{align}\label{pent eqn}
				\big[F^{fcd}_e \big]_{g \ell} \big[ F^{ab\ell}_d\big]_{fk} 
				&= \sum_h \big[ F^{abc}_g \big]_{fh} \big[ F^{ahd}_e \big]_{gk} \big[ F^{bcd}_k \big]_{h\ell}.
			\end{align}
			\item \textit{(Triangle Axiom.)} For every quadruple $(a, b, c, d) \in \ccat^4$ with $\mathbf{1} \in \{a, b, c\}$,
			$$F^{abc}_d = I.$$
			\item \textit{(Rigidity.)} For any $a \in \ccat$, 
			$\big[F^{a, a^*, a}_{a}\big]_{\mathbf{1} \mathbf{1}} 
			= \big[\big(F^{a^*, a, a^*}_{a^*}\big)^{-1}\big]_{\mathbf{1} \mathbf{1}}$.
		\end{enumerate} 

		\item There is a \textit{braiding}, \textit{$3j$-system}, or collection of \textit{$R$-symbols} $R \colon \ccat^3 \to \mathbb{C}$ written as $R(a, b, c) = R^{ab}_c$ satisfying the following conditions.
			\begin{enumerate}
			\item If $(a, b, c) \in \ccat^3$ is admissible, then $R^{ab}_c \neq 0$.
			\item \textit{(Hexagon Axiom.)} For every sextuple $(a, b, c, d, e, g) \in \ccat^6$, 
			\begin{equation}\label{hex eqn}
				\left(R^{ac}_e\right)^{\pm 1} \big[ F^{acb}_d \big]_{eg} \left(R^{bc}_g\right)^{\pm 1} 
				= \sum_f \big[ F^{cab}_d \big]_{ef} \, \left(R^{fc}_d\right)^{\pm 1} \, \big[ F^{abc}_d \big]_{fg}.
			\end{equation}
			\end{enumerate}

		\item There is a compatible \textit{pivotal structure} $t\colon \ccat \to \{\pm 1\}$ such that $t_{\mathbf{1}} = 1$ and for every admissible triple $(a, b, c) \in \ccat^3$,
		$$
			t_a t_b t_c = 
				\big[F^{a, b, c^*}_{\mathbf{1}}\big]_{a^*, c} 
				\big[F^{b, c^*, a}_{\mathbf{1}}\big]_{a^*, a} 
				\big[F^{c^*, a, b}_{\mathbf{1}}\big]_{b^*, b} 
		$$
	\end{enumerate}
\end{dfn}
\begin{rmk}
	The \textit{multiplicity-free} assumption is widely adopted in the literature. 
	An important implication of this hypothesis is that all braiding operators are semi-simple. Fortuitously, many anyon systems that are physically relevant or universal for computation are indeed multiplicity-free. For instance, the multiplicity-free $SO(n)_2$ system considered in \cite{cui_wang_2015} is physically relevant; the Fibonacci anyon system studied in \cite{trebst_troyer_wang_ludwig_2008} is universal for computation. The so-called \textit{metaplectic anyons} in $SO(n)_2$ are related to the representation theory discussed in \crossrefglnch\ and \crossrefsonch; in particular, \crossrefson{uqodn uqprime duality} describes the fusion of the spin object $S \in SO(n)_2$. No explicit $6j$-system is known for a non-multiplicity-free theory,\footnote{Private communication with Eric Rowell.} so not much can be said about such anyons with regards to the quantum computing application. 
\end{rmk}
	
A few remarks are in order. Observe that a fusion rule expresses the fusion of two anyons as some sort of super-position over all possible particle types. Commutativity  is equivalent to the identity
$$N^{ab}_c = N^{ba}_c,$$
while associativity is expressed by
$$\sum_{d \in \ccat} N^{ab}_d N^{dc}_e = \sum_{d \in \ccat} N^{bc}_d N^{ad}_e.$$
The conjugation axioms imply that $C_{ab} = N^{ab}_{\mathbf{1}}$ is an order $2$ permutation matrix. They are equivalent to 
$$N^{ab}_{\mathbf{1}} = \delta_{a^*, b} = N^{ba}_{\mathbf{1}}.$$
Of course a fusion ring \textit{isomorphism} is a bijective set function $\varphi\colon\ccat \to \ccat'$ preserving the fusion rules:
$$N^{\varphi(a), \varphi(b)}_{\varphi(c)} = N^{ab}_c, \qquad a, b, c \in \ccat.$$

\subsection{Anyon systems are RFCs}
\label{anyons are rfcs}

In this work we take an algebraic point of view: up to equivalence, anyon systems are in one-to-one correspondence with RFCs. 

To begin, we construct an RFC using the data defining a given anyon system $\ccat$. By slight abuse of notation, we denote this RFC also by $\ccat$. The simple objects in $\ccat$ are the labels in the anyon system. The semisimple structure is supplied by the fusion rules: the category $\ccat$ is a strict tensor category whose Grothendieck ring matches the fusion ring exactly. The Pentagon Axiom \eqref{pent eqn} enforces a compatibility that is equivalent to the MacLane coherence axiom, illustrated in \Cref{pentagon}, required to hold in any monoidal category. The $R$-symbols furnish the braiding on $\ccat$ and the pivotal structure is used to define a family of twists. Much like the Pentagon Axiom, the Hexagon Axiom \eqref{hex eqn} enforces a compatibility equivalent to the Hexagon \crossrefrfc{hexagon1,hexagon2} that must hold in any braided cateogry. To see this, we must understand the hom-sets of $\ccat$.

The category $\ccat$ is enriched over $\mathbb{C}$ and the morphism spaces are spanned by oriented framed tangle diagrams that allow for trivalent vertices to account for fusion. We construct the hom-spaces inductively as follows. To each fusion product $a \times b$ we assign \textit{fusion spaces} $\Hom(a \times b, c) = V^{ab}_c$ of dimension $N^{ab}_c$ for every $c \in \ccat$. Since we assume our anyon system is multiplicity-free, each fusion space $V^{ab}_c$ is at most one-dimensional. The \textit{splitting space} $\Hom(c, a \times b) = V^c_{ab}$ is the dual of $V^{ab}_c$. For each admissible triple $(a, b, c) \in \ccat$, we choose a basis of $V^{ab}_c$ and label its single non-zero vector by the following \textit{fusion tree}, with all strands oriented downwards:
\begin{equation*}
\begin{tikzpicture}[scale=0.4]
	\fusiontree{2}{{0,0}}
	\outerlab{{a,b}}{c}
	\node[scale=2.] at (8, -1.5) {$\in V^{ab}_c = \Hom(a \times b, c)$};
\end{tikzpicture}
\end{equation*}
Whenever $(a, b, c)$ is admissible we label the basis of the dual space $\Hom(c, a \times b)$ using the same fusion tree but with all strands pointing upwards. 

To the fusion of $m$ anyons $a_1 \times \cdots \times a_m$, we assign fusion spaces $\Hom(a_1 \times \cdots \times a_m, c) = V^{a_1,\ldots, a_m}_c$ that decompose into tensor products of fusion spaces of two anyons by matching intermediate indices:
\begin{equation}\label{m anyon decomp}
	V^{a_1, \ldots, a_m}_c 
	\cong \bigoplus_{e_1, \ldots, e_{m-2}} 
	V^{a_1 a_2}_{e_1} \otimes V^{e_1 a_3}_{e_2} 
	\otimes \cdots \otimes 
	V^{e_{m-3}, a_{m-1}}_{e_{m-2}} \otimes V^{e_{m-2}, a_m}_c.	
\end{equation}
This means there is a basis of $\Hom(a_1 \times \cdots \times a_m, b)$ enumerated by \textit{admissible} fusion trees: those whose every trivalent vertex is admissible. \Cref{std basis diagram} illustrates the \textit{standard basis} of $\Hom(a_1 \times \cdots \times a_m, b)$.
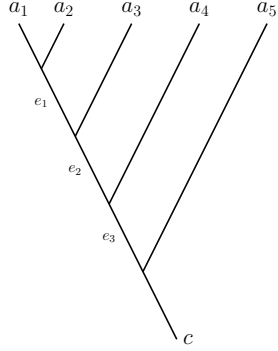
\begin{figure}[!h]
\begin{center}
\begin{tikzpicture}[scale=0.3]
	\stdtree{8}{{0,0}}
\end{tikzpicture}
\caption{The \textit{standard basis} of $\Hom(a_1 \times \cdots \times a_m, b) = V^{a_1, \ldots, a_m}_b$ is labeled by fusion trees as above. Here $m = 5$.}
\label{std basis diagram}
\end{center}
\end{figure}
Since the fusion rule is associative, many decompositions as in \Cref{m anyon decomp} are possible. For instance, if $m = 3$, then 
\begin{equation}\label{three anyon decomp}
V^{abc}_d \cong \bigoplus_e V^{ab}_e \otimes V^{ec}_d \cong \bigoplus_f V^{af}_d \otimes V^{bc}_f.	
\end{equation}

The Pentagon Axiom \eqref{pent eqn} ensures all possible decompositions of $\Hom(a_1 \times \cdots \times a_m, b)$ into tensor products are equivalent. This can be seen by using the $F$-matrix to define a family of isomorphisms between the two decompositions of $V^{abc}_d$ for every quadruple $(a, b, c, d) \in \ccat^4$. In particular, the $F$-matrix determines the change of basis isomorphism $\big[ F^{abc}_d \big]\colon \Hom\left((a \times b) \times c, d\right) \rightarrow \Hom\left(a \times (b \times c),  d\right)$ between the decompositions in \Cref{three anyon decomp}. This change of basis is known as an \textit{$F$-move} and it is depicted in \Cref{f move}; in bra-ket notation, it is written as
\[
	\ket{a,b; e} \ket{e,c; d} 
	= \sum_f \big[ F^{abc}_d \big]_{ef} \ket{a,f;d}\ket{b,c; f}.
\]
\begin{figure}[!h]
\begin{center}
\begin{tikzpicture}[scale=0.35]	
	\coordinate (tlc) at (0,3);
	\fusiontree{3}{tlc}
	\outerlab{{a,b,c}}{d}
	\innerlab{{e}}
		
	\draw[-latex] (6, 0) -- ++(3,0);
	\node[scale=1.75] at (12,-0.25) {
		$\displaystyle \sum_f F^{abc}_{d; \, ef}$
	};
	
	\coordinate (tlc) at (13.5,3);
	\fusiontree{3}{tlc}
	\outerlab{{a,b,c}}{d}
	\innerlab{{\relax}}
	\coordinate (p) at ($(3.1,-3) + (tlc)$);
	\node at (p) [replace] {$f$};
	\draw[very thick, white] ($(2,0) + (tlc)$) -- ++($\ss*(mop)$);
	\draw[semithick] ($(2,0) + (tlc)$) -- ++($\ss*(m)$);
\end{tikzpicture}
\end{center}	
\caption{The $F$-move is an isomorphism $\bigoplus_e V^{ab}_e \otimes V^{ec}_d \cong \Hom\left(a \times b \times c, d\right) \cong \bigoplus_f V^{af}_d \otimes V^{bc}_f$.}
\label{f move}
\end{figure}
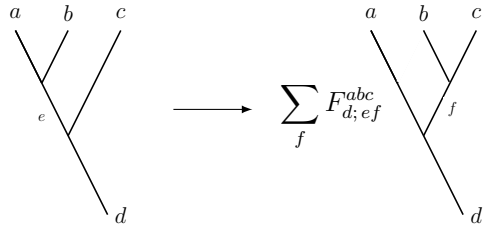

The Pentagon Axiom \eqref{pent eqn} enforces a consistency that ensures any two sequences of $F$-moves starting and ending in the same decomposition yield the same result; explicitly, \Cref{pent eqn} requires the equivalence of the two sequences illustrated in \Cref{tree pentagon}. The diagram in \Cref{tree pentagon} is the image of that in \Cref{pentagon} under the contravariant functor $\Hom(\cdot, e)$. Thus the MacLane Coherence Theorem \cite[Section~VII.2]{maclane_1988} implies the equivalence of all sequences of $F$-moves starting and ending in the same decomposition: for every diagram at the fusion space level, there is a corresponding diagram at the level of $\mathrm{Ob}(\ccat)$ whose commutativity is guaranteed by a pentagon like the one in \Cref{pentagon}. Note the MacLane Coherence Theorem applies to $\ccat$ automatically because it is a \textit{strict} monoidal category.

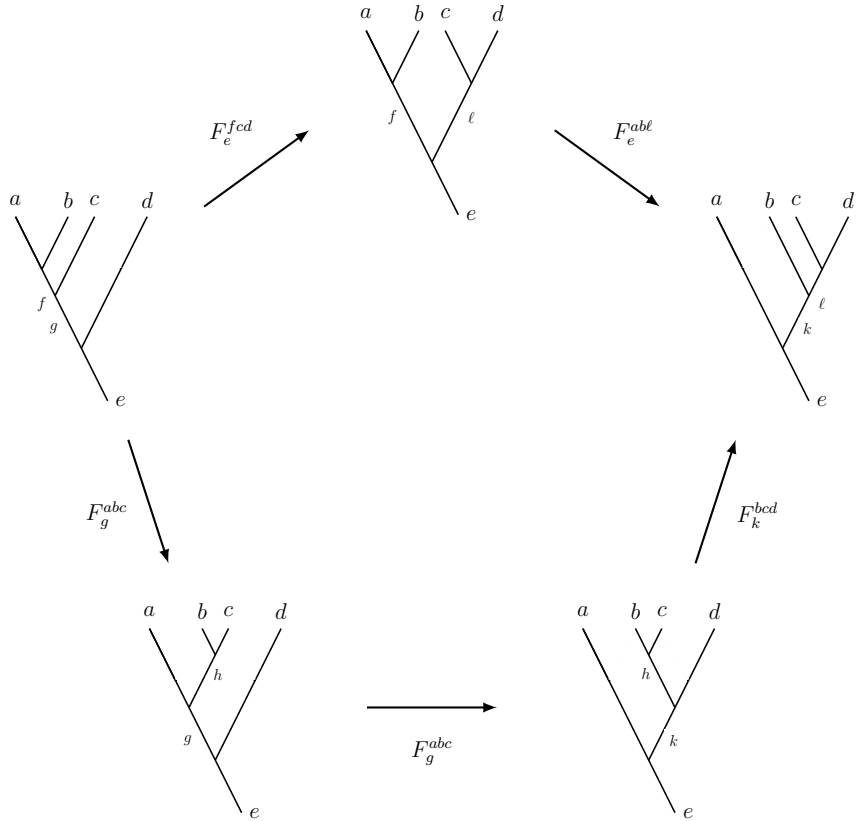
\begin{figure}[!h]
\begin{center}
\begin{tikzpicture}[scale=0.35]
	\tikzmath{\apo = 14;}
	\coordinate (shift) at (-2.5, 3);
	
	\coordinate (v) at ($(90:\apo) + (shift) - 4*(\deltay)$);
	\fusiontree{4}{v}
	\outerlab{{a, b, c, d}}{e}{v}
	\innerlab{{f,\ell}}
	
	\coordinate (v) at ($(90 + 72:\apo) + (shift)$);
	\fusiontree{4}{v}
	\outerlab{{a, b, c, d}}{e}
	\innerlab{{f,}}
	\node at ($(1.45, -4.25) + (v)$) [replace] {$g$};
	\draw[very thick, white] ($(3, 0) + (v)$) -- ++($\ss*(m)$);
	\draw[semithick] ($(3,0) + (v)$) -- ++($1.5*(mop)$);
	
	\coordinate (v) at ($(90 + 2*72:\apo) + (shift)$);
	\fusiontree{4}{v}
	\outerlab{{a, b, c, d}}{e}
	\innerlab{{,}}
	\node at ($(1.45, -4.25) + (v)$) [replace] {$g$};
	\node at ($(2.6, -1.75) + (v)$) [replace] {$h$};
	\draw[very thick, white] ($(3, 0) + (v)$) -- ++($\ss*(m)$);
	\draw[semithick] ($(3, 0) + (v)$) -- ++($1.5*(mop)$);
	\draw[very thick, white] ($(2, 0) + (v)$) -- ++($\ss*(mop)$);
	\draw[semithick] ($(2, 0) + (v)$) -- ++($0.5*(m)$);

	\coordinate (v) at ($(90 + 3*72:\apo) + (shift)$);
	\fusiontree{4}{v}
	\outerlab{{a, b, c, d}}{e}
	\innerlab{{,}}
	\node at ($(2.4, -1.7) + (v)$) [replace] {$h$};
	\node at ($(3.45, -4.25) + (v)$) [replace] {$k$};
	\draw[very thick, white] ($(2, 0) + (v)$) -- ++($\ss*(mop)$);
	\draw[very thick, white] ($(3, 0) + (v)$) -- ++($\ss*(m)$);
	\draw[semithick] ($(2, 0) + (v)$) -- ++($1.5*(m)$);
	\draw[semithick] ($(3, 0) + (v)$) -- ++($0.5*(mop)$);
		
	\coordinate (v) at ($(90 + 4*72:\apo) + (shift)$);
	\fusiontree{4}{v}
	\outerlab{{a, b, c, d}}{e}
	\innerlab{{,\ell}}
	\node at ($(3.45, -4.25) + (v)$) [replace] {$k$};
	\draw[very thick, white] ($(2, 0) + (v)$) -- ++($\ss*(mop)$);
	\draw[semithick] ($(2, 0) + (v)$) -- ($(3.5, -3) + (v)$);
	
	\tikzmath{\start = 0.35; \tar = 1 - \start;}
	\foreach[count=\next] \curr in {0,...,4} {
		\pgfmathsetmacro\ti{90 + int(\curr)*72}
		\pgfmathsetmacro\tf{90 + int(\next)*72}
		\coordinate (vi) at ($(\ti:\apo)!\start!(\tf:\apo)$);
		\coordinate (vf) at ($(\ti:\apo)!\tar!(\tf:\apo)$);
		\pgfmathparse{int(mod(\curr,4))}
		\ifnum\pgfmathresult=0
			\draw[thick, latex-] (vi) -- (vf);
		\else
			\draw[thick, -latex] (vi) -- (vf);
		\fi
	}
	
	\tikzmath{\apo = \apo-1; \factor=1.45;}
	\node[scale=\factor] at ($(90 + 0.5*72:\apo)$) {$F^{fcd}_e$};
	\node[scale=\factor] at ($(90 + 1.5*72:\apo)$) {$F^{abc}_g$};
	\node[scale=\factor] at ($(0, -\apo)$) {$F^{abc}_g$};
	\node[scale=\factor] at ($(90 + 3.5*72:\apo)$) {$F^{bcd}_k$};
	\node[scale=\factor] at ($(90 + 4.5*72:\apo)$) {$F^{ab\ell}_e$};
\end{tikzpicture}
\caption{The Pentagon Relations guarantee that different sequences of $F$-moves starting and ending in the same fusion basis give the same result. The Pentagon Axiom \eqref{pent eqn} are equivalent to the commutativity of this diagram.}
\label{tree pentagon}
\end{center}	
\end{figure}

\begin{figure}[!h]
\begin{center}
\scalebox{0.45}{
\begin{tikzpicture}[line width=1.1]
	\tikzmath{\apo = 7; \ss=1.7;}
	\node[scale=\ss] at ($(90:\apo)$) {$(a \times b) \times (c \times d)$};
	\node[scale=\ss] at ($(90 + 1*72:\apo)$) {$\big((a \times b) \times c\big) \times d$};
	\node[scale=\ss] at ($(90 + 2*72:\apo)$) {$\big(a \times ( b \times c)\big) \times d\quad\quad$};
	\node[scale=\ss] at ($(90 + 3*72:\apo)$) {$\quad\quad a \times \big((b \times c) \times d\big)$};
	\node[scale=\ss] at ($(90 + 4*72:\apo)$) {$a \times \big(b \times (c \times d)\big) $};
		
	\tikzmath{\start = 0.3; \tar = 1 - \start;}
	\foreach[count=\next] \curr in {0,...,4} {
		\pgfmathsetmacro\ti{90 + int(\curr)*72}
		\pgfmathsetmacro\tf{90 + int(\next)*72}
		\coordinate (vi) at ($(\ti:\apo)!\start!(\tf:\apo)$);
		\coordinate (vf) at ($(\ti:\apo)!\tar!(\tf:\apo)$);
		\pgfmathparse{int(mod(\curr,4))}
		\ifnum\pgfmathresult=0
			\draw[{Latex[length=3mm]}-] (vi) -- (vf);
		\else
			\draw[-{Latex[length=3mm]}] (vi) -- (vf);
		\fi
	}
	
	\tikzmath{\apo = \apo; \ss = 0.8*\ss; }
\end{tikzpicture}}
\end{center}
\caption{The Pentagon Axiom \eqref{pent eqn} is equivalent to the commutativity of the image of this diagram under $\Hom(\cdot, e)$.}
\label{pentagon}
\end{figure}
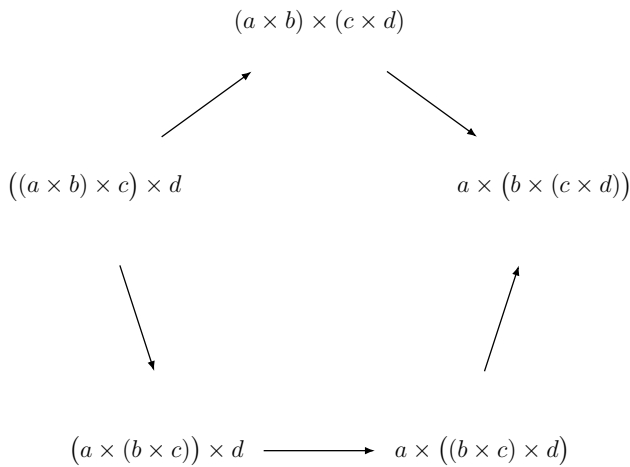

Since the morphism spaces are spanned by certain tangles, we may construct a braiding on $\ccat$ using the $3j$-system as follows. For starters, we define the \textit{$R$-move} on $\Hom(a \times b, c)$, for any admissible triple $(a, b, c) \in \ccat^3$, by identifying the following tangles:
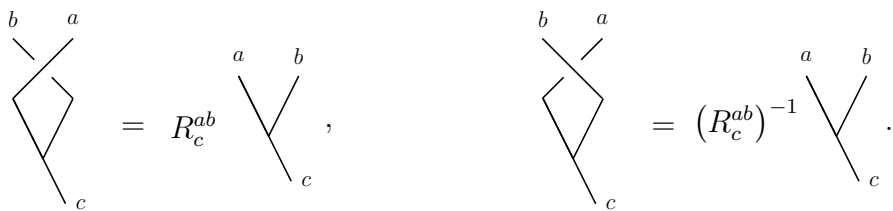
\begin{equation*}
\begin{tikzpicture}[scale=0.4]
	\coordinate (o);
	\braidedtree{2}{1}{o}
	\outerlab{{,}}{c}
	\node at ($(0,2) + (\deltay)$) [replace] {\Large$b$}; 

	\node[scale=2] at (4, -1) {$=$};
	\node[scale=2] at (6, -1) {$R^{ab}_c$};

	\fusiontree{2}{{7.5,0.75}}
	\outerlab{{a,b}}{c}
	
	\node[scale=2] at (10.5, -1) {,};
\end{tikzpicture}
\qquad\qquad\qquad
\begin{tikzpicture}[scale=0.4]
	\coordinate (o);
	\fusiontree{2}{o}
	\coordinate (lcoord) at ($1.5*(0.65, 0) - (1,0) + (o)$);
	\coordinate (rcoord) at ($(lcoord) + (step)$);
	\fill[color=white] ($(lcoord) + (\deltay)$) circle (1em);
	\node at ($(lcoord) + (0,2) + (\deltay)$) {\Large $a$};
	\fill[color=white] ($(rcoord) + (\deltay)$) circle (1em);
	\node at ($(rcoord) + (0,2) + (\deltay)$) {\Large $a$};
	\draw[semithick] (lcoord) -- ($(lcoord) + (step) + (0,2)$);
	\fill[color=white] ($(lcoord) + 0.5*(step) + (0,1)$) circle (1em);
	\draw[semithick] (rcoord) -- ($(rcoord) - (step) + (0,2)$);
	\outerlab{{,}}{c}
	\node at ($(0,2) + (\deltay)$) [replace] {\Large$b$}; 

	\node[scale=2] at (4, -1) {$=$};
	\node[scale=2] at (6.75, -0.75) {$\left(R^{ab}_c\right)^{-1}$};

	\fusiontree{2}{{8.75,0.75}}
	\outerlab{{a,b}}{c}	
	\node[scale=2] at (11.5, -1) {.};
\end{tikzpicture}
\end{equation*}
Note that these diagrams distinguish between a positive crossing and its inverse. Using these identifications, we define a natural isomorphism $R_{ab}\in \Hom(a \times b, b \times a) = \bigoplus_{c \in \ccat} \Hom(a \times b, c)$ satisfying
\begin{equation}\label{anyonic braiding op}
	R_{ab} = \sum_c N_{ab}^c R^{ab}_c \ket{ab; c}
\end{equation}
for each every pair $(a, b)$ of \textit{simple} objects in $\ccat$.

Now we extend the map $(a, b) \to R_{ab}$ to generic objects of $\ccat$ using its tensor category structure: the Hexagon Axiom \eqref{hex eqn} allows us to decompose terms of the form $R_{a, b \times c}$ and $R_{a \times b, c}$, for simple objects $a, b, c$, in a consistent manner using an appropriate sequence of $F$-moves. In particular, the Hexagon Axiom \eqref{hex eqn} guarantees the commutativity of the hexagon in \Cref{hexagon}, which is needed in any tensor category, because it explicitly enforces equivalence between the two sequences of $F$- and $R$-moves illustrated in \Cref{tree hex}. Thus, much like the Pentagon Axiom \eqref{pent eqn} enforces the coherence for the tensor product required to make $\ccat$ a monoidal category, the Hexagon Axiom \eqref{hex eqn} imposes the compatibility on the braiding needed to make $\ccat$ a braided category, in the sense of \crossrefrfc{braided tensor category}.

\begin{figure}
\centering
\begin{tikzcd}[column sep={2.4cm,between origins}, row sep={1.732050808cm,between origins}]
	& (a \times c) \times b \ar[rr] 
	&& a \times (c \times b) \arrow[rd, "\id_a \otimes R_{c, b}"] & 
	&  \\
	(c \times a) \times b \arrow[rd] \arrow[ru, "R_{c,a} \otimes \id_b"] 
	&  &&  & a \times(b \times c) \\
	& c \times (a \times b) \arrow[rr, "R_{c, a \times b}"] && (a \times b) \times c
	\ar[ru]
\end{tikzcd}
\caption{The Hexagon Axiom \eqref{hex eqn} guarantees the commutativity of this diagram, which is required in any braided category.}
\label{hexagon}
\end{figure}

\begin{figure}[!h]
\begin{center}
\begin{tikzpicture}[scale=0.33]
	\tikzmath{\apo = 14;}
	\coordinate (yshift) at (0,3);
	\coordinate (tshift) at (-2.5, -0.5);
		
	\coordinate (a);
	\coordinate (b) at ($(2,0) + (a)$);
	\coordinate (c) at ($(4,0) + (a)$);
	\coordinate (aa) at ($(a) + (yshift)$);
	\coordinate (bb) at ($(b) + (yshift)$);
	\coordinate (cc) at ($(c) + (yshift)$);
	
	\coordinate (v) at ($(0:\apo) + (tshift)$);
	\fusiontree{3}{v}
	\outerlab[{0,3}]{{a,b,c}}{d}
	\innerlab{{\relax}}
	\node at ($(b) + (1.15, -3) + (v)$) [replace] {$g$};
	\filldraw[very thick, white] ($(b) + (v)$) -- ++($\ss*(mop)$);
	\draw[semithick] ($(b) + (v)$) -- ++(m);
	\draw[semithick] ($(aa) + (v)$) -- ($(a) + (v)$);
	\draw[semithick] ($(bb) + (v)$) -- ($(b) + (v)$);
	\draw[semithick] ($(cc) + (v)$) -- ($(c) + (v)$);
		
	\coordinate (v) at ($(60:\apo) + (tshift)$);
	\fusiontree{3}{v}
	\outerlab[{0,3}]{{a,b,c}}{d}
	\innerlab{{\relax}}
	\node at ($(b) + (1.15, -3) + (v)$) [replace] {$g$};
	\filldraw[very thick, white] ($(b) + (v)$) -- ++($\ss*(mop)$);
	\draw[semithick] ($(b) + (v)$) -- ++(m);
	\draw[semithick] ($(aa) + (v)$) -- ($(a) + (v)$);
	\draw[semithick] ($(bb) + (v)$) -- ($(c) + (v)$);
	\filldraw[white] ($(bb)!0.5!(c) + (v)$) circle (1em);
	\draw[semithick] ($(cc) + (v)$) -- ($(b) + (v)$);
	
	\coordinate (v) at ($(120:\apo) + (tshift)$);
	\fusiontree{3}{v}
	\outerlab[{0,3}]{{a,b,c}}{d}
	\innerlab{{e}}
	\draw[semithick] ($(aa) + (v)$) -- ($(a) + (v)$);
	\draw[semithick] ($(bb) + (v)$) -- ($(c) + (v)$);
	\filldraw[white] ($(bb)!0.5!(c) +(v)$) circle (1em);
	\draw[semithick] ($(cc) + (v)$) -- ($(b) + (v)$);
	
	\coordinate (v) at ($(180:\apo) + (tshift)$);
	\fusiontree{3}{v}
	\outerlab[{0,3}]{{a,b,c}}{d}
	\innerlab{{e}}
	\draw[semithick] ($(aa) + (v)$) -- ($(b) + (v)$);
	\draw[semithick] ($(bb) + (v)$) -- ($(c) + (v)$);
	\filldraw[white] ($(aa)!0.666!(b) + (v)$) circle (1em);
	\filldraw[white] ($(bb)!0.333!(c) +(v)$) circle (1em);
	\draw[semithick] ($(cc) + (v)$) -- ($(a) + (v)$);
	
	\coordinate (v) at ($(240:\apo) + (tshift)$);
	\fusiontree{3}{v}
	\outerlab[{0,3}]{{a,b,c}}{d}
	\innerlab{{\relax}}
	\node[semithick] at ($(b) + (1.15, -3) + (v)$) [replace] {$f$};
	\filldraw[very thick, white] ($(b) + (v)$) -- ++($\ss*(mop)$);
	\draw[semithick] ($(b) + (v)$) -- ++(m);
	\draw[semithick] ($(aa) + (v)$) -- ($(b) + (v)$);
	\draw[semithick] ($(bb) + (v)$) -- ($(c) + (v)$);
	\filldraw[color=white] ($(aa)!0.666!(b) + (v)$) circle (1em);
	\filldraw[color=white] ($(bb)!0.333!(c) + (v)$) circle (1em);
	\draw[semithick] ($(cc) + (v)$) -- ($(a) + (v)$);
	
	\coordinate (v) at ($(300:\apo) + (tshift)$);
	\fusiontree{3}{v}
	\outerlab[{0,3}]{{a,b,c}}{d}
	\innerlab{{f}}
	\draw[semithick] ($(aa) + (v)$) -- ($(a) + (v)$);
	\draw[semithick] ($(bb) + (v)$) -- ($(b) + (v)$);
	\draw[semithick] ($(cc) + (v)$) -- ($(c) + (v)$);	
	
	\tikzmath{\start = 0.35; \tar = 1 - \start;}
	\foreach \curr in {0,...,5} {
		\pgfmathsetmacro\ti{int(\curr)*60}
		\coordinate (vi) at ($(\ti:\apo)!\start!(\ti+60:\apo)$);
		\coordinate (vf) at ($(\ti:\apo)!\tar!(\ti+60:\apo)$);
		\ifnum\curr<3
			\draw[thick, latex-] (vi) -- (vf);
		\else
			\draw[thick, -latex] (vi) -- (vf);
		\fi
	}
	
	\tikzmath{\apo = \apo; \factor=1.5;}
	\node[scale=\factor] at ($(0.5*60:\apo)$) {$R^{cb}_g$};
	\node[scale=\factor] at ($(1.5*60:\apo-4)$) {$F^{acb}_d$};
	\node[scale=\factor] at ($(2.5*60:\apo)$) {$R^{ca}_e$};
	\node[scale=\factor] at ($(3.5*60:\apo-4.5)$) {$F^{cab}_d$};
	\node[scale=\factor] at ($(4.5*60:\apo)$) {$R^{cf}_d$};
	\node[scale=\factor] at ($(5.5*60:\apo-4.5)$) {$F^{abc}_d$};
\end{tikzpicture}
\end{center}	
\caption{Enforcing the Hexagon Axiom \Cref{hex eqn}, which is equivalent to the commutativity of this diagram, ensures that the braiding on $\ccat$ is compatible with its monoidal structure.}
\label{tree hex}
\end{figure}
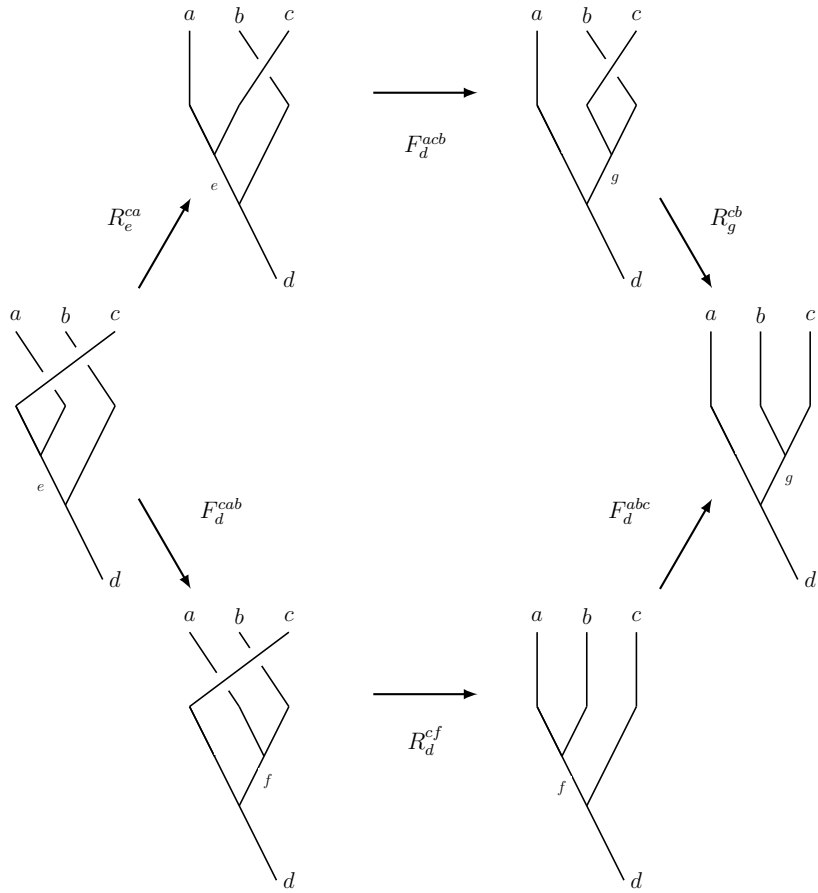

With compatible braiding and monoidal structures in hand, it remains to define a twist. To do this, we must first construct a left rigidity on $\ccat$, in the sense of \crossrefrfc{left duality}. The conjugation map $\ast\colon \ccat \to \ccat$ on the anyon system supplies the dual objects; concretely, for each $a \in \mathrm{Ob}(\ccat)$ we let $a^* \in \mathrm{Ob}(\ccat)$ denote its dual. In addition, we define the birth and death duality maps $b_a \in \Hom(\mathbb{C}, a \times a^*)$ and $d_a \in \Hom(a^* \times a, \mathbb{C})$ as the tangles
\begin{center}
\begin{tikzpicture}
	\node at (0, 0) {$b_a = $};
	\coordinate (blc) at (0.75,-0.25);
	\draw[semithick,->] ($(0,0.5)+(blc)$) to ($(0,0.25)+(blc)$) to [out=-90, in=180] ($(0.5,-0.25)+(blc)$) to [out=0, in=-90] ($(1,0.25)+(blc)$) to ($(1,.5)+(blc)$);
	\node at ($(0.05, 0.85) + (blc)$) {$a^{\phantom *}$};
	\node at ($(1.05, 0.85) + (blc)$) {$a^*$};
	\node at (4, 0) {and};
	\node at (7, 0) {$d_a = $};
	\coordinate (blc) at (7.75,-0.75);
	\draw[semithick,<-] ($(0,0.5)+(blc)$) to ($(0,0.75)+(blc)$) to [out=90, in=180] ($(0.5,1.25)+(blc)$) to [out=0, in=90] ($(1,0.75)+(blc)$) to ($(1,0.5)+(blc)$);
	\node at ($(0.05, 0.15) + (blc)$) {$a^*$};
	\node at ($(1.05, 0.15) + (blc)$) {$a^{\phantom *}$};
\end{tikzpicture} 
\end{center}
The rigidity axioms \crossrefrfc{left rigidity} follow immediately from the so-called \textit{arms-bending} Moves 5-8 in \crossrefrfc{ribbon category relations}.

We may now define a family of twists on $\ccat$ using the left duality and the pivotal structure $t_a$. First, for each $a \in \ccat$, define the isomorphism $\psi_a\colon a^{**} \to a$ as the composition $\psi_a = (\id_a \otimes d_{a^*})(c_{a, a^{**}} \otimes \id_{a^*})(\id_{a^{**}} \otimes b_a)$ corresponding to the following tangle:
\begin{equation*}
\begin{tikzpicture}[scale=0.5]
	\coordinate (blc) at (-10, 0);
	
	\node[scale=1.2] at ($(-2, 4.5) + (blc)$) {$\psi_a = $};
	\draw[semithick,->] ($(0,9) + (blc)$) -- ++(0,-3);
	\draw[semithick,->] ($(0,7)+(blc)$) to ($(0,5.5)+(blc)$) to [out=-90, in=90] ($(2.5,3.5)+(blc)$) to ($(2.5,2.5)+(blc)$);
	\draw[semithick] ($(2.5,2.5)+(blc)$) to ($(2.5,2)+(blc)$) to [out=-90, in=180] ($(3.85,1)+(blc)$) to [out=0, in=-90] ($(5,2)+(blc)$) to ($(5,6)+(blc)$) to [out=90, in=-90] ($(5,7)+(blc)$) to [out=90, in=0] ($(3.75,8)+(blc)$);
	\path[fill=white] ($(1.25,4.5)+(blc)$) circle (.3);
	\draw[semithick,->] ($(5,4)+(blc)$) -- ++(0,1);
	\draw[semithick,->] ($(3.75,8)+(blc)$) to [out=180, in=90] ($(2.5,7)+(blc)$) to ($(2.5,5.5)+(blc)$) to [out=-90, in=90] ($(0,3.5)+(blc)$) to ($(0,2.5)+(blc)$);
	\draw[semithick,->] ($(2.5,7)+(blc)$) -- ++(0,-1);
	\draw[semithick] ($(0,3.5)+(blc)$) to ($(0,0)+(blc)$);
\end{tikzpicture}
\end{equation*}
Since the conjugation map is an involution, we have $a^{**} = a$ in the tensor category $\ccat$. Now set $\theta_a = t_a \psi_a$ for each $a \in \ccat$. Theorem~4.17 in \cite{wang_2010} guarantees that  $\theta\colon \ccat \to \ccat$ extends to a functor defining a family of twists in the sense of \crossrefrfc{ribbon cat}. Alternatively, see Lemma~2.2.2 in \cite{baki}. Notice that the tangle defining $\psi_a$ motivates the RFC terminology: multiplying by $\theta_a$ introduces a twist in the framed strand, or ribbon, labeled by $a$. 

Since we may construct a category equipped with compatible monoidal, braiding, and ribbon structures using the data in \Cref{anyon system} characterizing an anyon system, it follows that each anyon system determines an RFC. 

In fact this RFC is characterized uniquely by the corresponding anyon system up to equivalence of its $6j$-system \cite[Proposition~1.1]{yamagami_2002}. 

\begin{dfn}\cite[Definition~4.10]{wang_2010}\label[defn]{anyon gauge transform}
	Two $6j$-systems $F$ and $\widetilde{F}$ on a label set $\ccat$ are \textit{gauge equivalent} if there exists a map $f\colon \ccat^3 \to \mathbb{C}$, written as $f(a, b, c) = f^{ab}_c$ and known as a \textit{gauge transformation}, satisfying the following axioms. 
	\begin{enumerate}
		\item $f^{ab}_c \neq 0$ if and only if $(a, b, c)$ is admissible.
		\item $f^{\mathbf{1} a}_a = f^{a \mathbf{1}}_a = 1$ for every $a \in \ccat$. 
		\item \textit{(Rectangle Axiom.)} For every sextuple $(a, b, c, d, x, y) \in \ccat^6$, 
		\begin{equation}\label{rect axiom}
			f^{bc}_y  f^{ay}_d F^{abc}_{d; xy} = \widetilde{F}^{abc}_{d; xy} f^{ab}_x f^{xc}_d.
		\end{equation}
	\end{enumerate}
	Then $F$ and $\widetilde{F}$ are \textit{equivalent} if they are gauge equivalent up to fusion ring automorphism.
\end{dfn}

We may interpret the gauge transformation $\{f^{ab}_c\}$ as a family of change-of-basis maps on $\{V^{ab}_c\}$ that is compatible with the $F$-matrix. Scalars suffice because we only consider multiplicity-free anyon systems, which means each fusion space $\Hom(a \times b, c) = V^{ab}_c$ is at most one-dimensional. Thus \Cref{rect axiom} expresses the commutativity of the following diagram.
\begin{equation*}
\begin{tikzcd}[row sep=large]
	\bigoplus_x V^{ab}_x \otimes V^{xc}_d \ar[r, "="] 
		\ar[d, swap, "f^{ab}_x \otimes f^{xc}_d\quad"]
	& \Hom\left(a \times b \times c, d\right) \ar[r, "F^{abc}_d"] 
	& \bigoplus_y V^{ay}_d \otimes V^{bc}_y \ar[d, "\quad f^{ay}_d \otimes f^{bc}_y"]
	\\
	\bigoplus_x V^{ab}_x \otimes V^{xc}_d \ar[r, "\widetilde{F}^{abc}_d"]
	& \Hom\left(a \times b \times c, d\right) \ar[r, "="]
	& \bigoplus_y V^{ay}_d \otimes V^{bc}_y 
\end{tikzcd}
\end{equation*}

Since there is a one-to-one correspondence between anyon systems, up to gauge equivalence, and RFCs, up to categorical equivalence, from now on we do not distinguish between an anyon system and the RFC it characterizes.

To conclude, we remark that the fusion rule alone is \textit{almost} sufficient to pin down the corresponding RFC: \textit{Ocneanu rigidity} states that there are only finitely many equivalence classes of RFCs with a given fusion rule. Ocneanu himself never published a proof but there are various secondary sources; see, e.g., \cite[Theorem~2.28]{etingof_fusion}, \cite[Section~E.6]{kitaev_2006}, and \cite[Theorem~4.1]{hagge}.

\subsection{RFC framework for TQC}
\label{rfc framework for tqc}

Now we explain how to leverage the RFC framework to simulate anyonic quantum computers. The key is to look at the braid representations induced by the RFC corresponding to a given anyon system under the right light.

These representations arise as follows. For concreteness, let $\ccat$ denote an anyon system, let $a$ denote any label in $\ccat$, and consider the map $R_{aa}$ defined by \Cref{anyonic braiding op}. Then for any $m$ we obtain a braid group representation $B_m \to \End_{\ccat}(a^{\times m})$ satisfying 
\begin{equation}\label{braid repn}
	\sigma_j \to \id_{a^{\times (j-1)}} \otimes R_{aa} \otimes \id_{a^{\times (m - j-1)}},
\end{equation}
with $\sigma_j$ denoting the $j$th braid generator of $B_m$. For details, recall \crossrefrfc{braid repn induced by braiding}.

\Cref{braid repn} may be understood as an action of $B_m$ on the fusion spaces $\Hom(a^{\times m}, b)$, for $b \in \ccat$. This follows from the semisimplicity of $\ccat$: if we write $a^{\times m} = \bigoplus_{b \in \ccat} N_b \, b$ for some multiplicities $N_b$, we see that
\begin{equation}\label{end a^m hom decomp}
	\End(a^{\times m}) = \bigoplus_{b \in \ccat} N_b \Hom(a^{\times m}, b).
\end{equation}

This action is essential: it explains the quantum computing application. To see this, we must borrow two facts from theoretical physics. The first is that anyon configurations are constrained to two spatial dimensions. This means the world-lines tracking the dynamics of point-like anyons in their $(2+1)$-dimensional spacetime are described by braids, since the fundamental group of the configuration space associated to $m$ well-separated anyons in a plane is the braid group $B_m$. The second fact is that anyon configurations in certain topological phases of matter are expected to support degenerate ground states protected by a positive energy gap \cite[Section~2]{nayak_simon_stern_freedman}; these anyons are said to be \textit{non-abelian}. This means we can assign a full $n$-dimensional Hilbert space, with $n > 1$, to a fixed anyon configuration, and that it is possible to exchange anyons adiabatically while preserving this state space.

The most remarkable feature is that, while preserving the state \textit{space}, each exchange induces a state transformation that is non-trivial in general. Since the world-lines describing anyon exchanges in the corresponding $(2+1)$-dimensional spacetime determine (topological) equivalence classes of braids, there is a correspondence between exchanges and unitary operators on the state space that factors through the braid group. In particular, each operator exchanging anyons with adjacent labels corresponds to a braid generator.
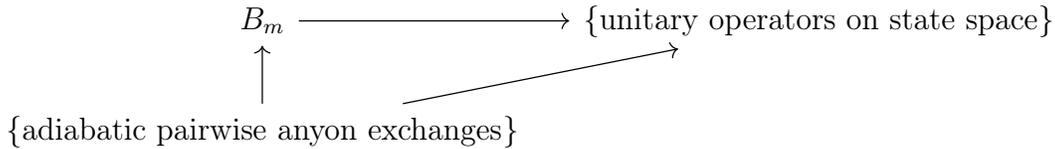
\begin{equation*}
\begin{tikzcd}[column sep=0.5cm]
	B_m \rar
	& \{\text{unitary operators on state space}\} \\
	\{\text{adiabatic pairwise anyon exchanges}\} \uar \ar[ur]
\end{tikzcd}
\end{equation*}

This gives the endomorphism spaces of $\ccat$ a nice physical interpretation: each tangle in $\End(a^m)$ models (a topological equivalence class of) world-lines in a $(2+1)$-dimensional spacetime describing the time evolution of a certain anyon configuration.

Fusion spaces enter the story as follows. Following a sequence of anyon exchanges, which transforms a given state in the associated state space according to a representation of $B_m$, the anyons may be fused together as a kind of measurement operation. The fusion, which we denote by $a^{\times m}$, results in a superposition $\bigoplus_{b \in \ccat} N_b \, b$ over all possible anyon types. The probability of observing any particular outcome is proportional to the multiplicity $N_b$. Therefore the state space associated to our given anyon configuration decomposes into a sum of fusion spaces according to the possible measurement outcomes. 

Thus the braid group action induced by $\ccat$ on the fusion space $\Hom(a^{\times m}, b)$ models the physical process of adiabatic anyon exchange followed by fusing $a^{\times m}$ to obtain $b$. 

This physical process results in computation once we identify a distinguished \textit{computational basis} $\ket{0}, \ldots, \ket{d-1}$ in $\Hom(a^{\times m}, b)$ and keep track of anyon exchanges using the braid group representation described explicitly in terms of our distinguished basis. The representation $\rho^a_m\colon B_m \to \End_{\ccat}(a^{\times m})$ is a bookkeeping device: the image of each generator tracks how a given vector in the chosen qudit computational basis $\ket{0}, \ldots, \ket{d-1}$ transforms under the effect of an anyon exchange. In other words, the representation provides a dictionary between braid generators and unitary transformations on the degenerate ground state. \Cref{sigma gen on fusion tree} illustrates this process.

\def\nstrands{6}
\def\braididx{2}
\begin{figure}[!h]
\begin{center}
\begin{tikzpicture}[scale=0.4]
	\node[scale=1.25] at (0, 0.55) {$\sigma_\braididx \ket{0} = $};
	
	\braidedstdtree{\nstrands}{\braididx}{{3,\M/2}}
	\node[replace] at ($(0.5,0) + (bot)$) {\Large $b$};
	\node[replace] at ($(tlc) + (\deltay)$) {\Large $a$};
	\node[replace] at ($(tlc) + (\deltay) + (8, 0)$) {\Large $a$};
	\draw[semithick, -latex] (1.5*\nstrands + 3, 1) -- ++(3,0);
	\node[scale=2.25] at (1.5*\nstrands + 11, 1+-0.45) {
	$
	\displaystyle \sum_{f} 
	[\sigma_\braididx]_{\ket{e_1 e_2}, \; \ket{f e_2}}
	$
	};
	\stdtree{\nstrands}{{23,\M/2+0.5}}
	\innerlab{{f,\relax, \relax, e_2}}
	\node[replace] at ($(tlc) + (\deltay)$) {\Large $a$};
	\node[replace] at ($(tlc) + (\deltay) + (2, 0)$) {\Large $a$};
	\node[replace] at ($(tlc) + (\deltay) + (5, 0)$) {\Large $a$};
	\node[replace] at ($(tlc) + (\deltay) + (8, 0)$) {\Large $a$};
	\node[replace] at ($(0.5,0) + (bot)$) {\Large $b$};
\end{tikzpicture}
\end{center}	
\caption{The action of $\sigma_\braididx \in B_m$ on a chosen standard basis state $\ket{0} \in \Hom(a^{\times m}, b)$, with $m = 4$.}
\label{sigma gen on fusion tree} 
\end{figure}

Thus depending on $\ccat$, $a$, and $m$, it may be possible to identify traditional quantum circuit gates, like the CNOT and single qubit rotations, as certain braid words in the image of $\rho^a_m$. \Cref{numerical braids} illustrates some explicit identifications.

In general, the image of $\rho^a_m$ describes the set of possible logic gates available to an anyonic quantum computer processing information encoded in the fusion spaces $\Hom(a^{\times m}, b)$ by braiding anyons of type $a$. In this way, certain RFCs provide a means for modeling anyonic quantum computers.

Constructing these braid representations is therefore essential for understanding the computational power of a given anyon system. The following section obtains formulas for these representations in terms of $F$- and $R$-matrices.

\section{Braiding via $F$- and $R$-matrices}
\label{braiding via frt mats}
Let $\ccat$ be an anyon system in the sense of \Cref{anyon system}. In this section we describe the braid group representations $B_m \to \End_{\ccat}(a^{\times m})$ induced by $\ccat$ via \Cref{braid repn}. We decompose $\End(a^{\times m})$ as in \Cref{end a^m hom decomp} and describe the representations explicitly with respect to the computational bases defined in \Cref{comp basis}. Then we obtain the necessary matrix coefficients using the $F$- and $R$-matrices defining $\ccat$ in \Cref{general braiding formula}.

{\sc SageMath} can compute the explicit representations described below, for those anyon systems described by its \texttt{FusionRing} class; precisely, these are the anyon systems whose corresponding RFC is a semisimple quotient of the category of finite dimensional representations of a quantum group associated to a complex simple Lie algebra with deformation parameter a root of unity.

As noted in see \Cref{abcs of tqc}, these representations are critical to the quantum computing application because they describe the sets of possible logic gates that an anyonic quantum computer can use to process information encoded in associated fusion spaces.

To begin, we define a preferred computational basis $\mathcal{T}_m$ for each fusion space $\Hom(a^{\times m}, b)$. Our basis corresponds to the decomposition of $\Hom(a^{\times m}, b)$ obtained by fusing the anyons in pairs; for instance, if $m = 2r + 1$ we have
\[
	\Hom(a^{\times m}, b) 
	\cong \bigoplus_{\substack{t_1, \ldots, t_r \\ \ell_1, \ldots, \ell_{r-1}}} 
	V^{aa}_{t_1} \otimes \cdots \otimes V^{aa}_{t_r} 
	\otimes 
	V^{t_1 t_2}_{\ell_1} \otimes V^{\ell_1 t_3}_{\ell_2} 
	\otimes \cdots \otimes 
	V^{\ell_{r - 1}, a}_b
\]
This basis is parametrized by the $m-2$ labels $\ket{t \ell}$. For concreteness, \Cref{tree basis pic} depicts basis elements of $\Hom(a^7, b)$ and $\Hom(a^8, b)$ respectively labeled by
\begin{equation*}
	\ket{t_1 t_2 t_3 \ell_1 \ell_2} 
	\quad \text{and} \quad 
	\ket{t_1 t_2 t_3 t_4 \ell_1 \ell_2}.
\end{equation*}

\begin{figure}[!h]
\begin{center}
\begin{tikzpicture}[scale=0.35]
	\fusiontree{7}{{0,0}}
	\fusiontree{8}{{16,0}}
\end{tikzpicture}
\end{center}	
\caption{A fusion tree basis $\ket{t \ell}$ for $\Hom(a^{\otimes m}, b)$. On the left, $m = 7$ is odd. On the right, $m = 8$ is even.}
\label{tree basis pic}
\end{figure}
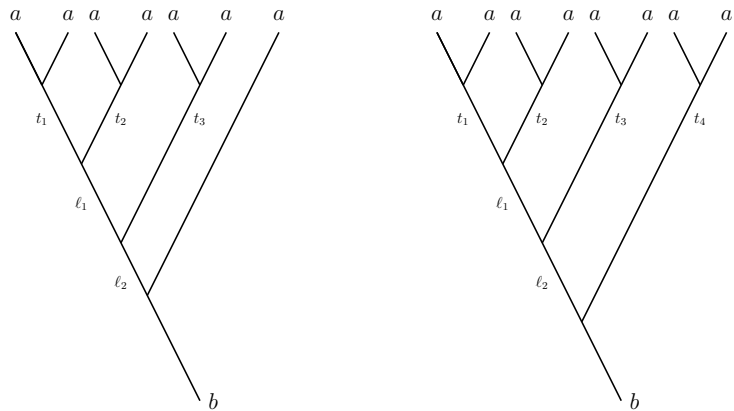

The next proposition enumerates $\mathcal{T}_m$ recursively. The bases differ slightly depending on the parity of $m$, so we treat the two cases separately. Our implementation in {\sc SageMath} mimics this construction. 
\begin{prop}\label[prop]{comp basis}
If $m = 2r$, let
\[
\mathcal{T}_m = 
	\bigcup_{t \in \{a \times a\}^r} 
	\{
		\ket{t \ell}
		\mid 
		\ell_1 \in \{t_1 \times t_2\}, \; 
		\ell_{j+1} \in \{\ell_j \times t_{j+2}\}, \; 
		j=1, \ldots, r-3, \; 
		b \in \{\ell_{r-2} \times t_r\}
	\}.
\]
If $m = 2r + 1$, replace the condition $b \in \{\ell_{r-2} \times t_r\}$ defining $\mathcal{T}_m$ by $b \in \{\ell_{r-1} \times a\}$. For any $m \geq 1$, $\mathcal{T}_m$ defines a basis of $\Hom(a^{\otimes m}, b)$.
\end{prop}
\begin{rmk}
When $m = 2r+1$ is odd, there is an ``unpaired'' $a$ that fuses with $\ell_{r-1}$  to produce the root $b$, so the tree labeled by $\ket{t \ell}$ is admissible only if $b \in \ell_{r - 1} \times a$.
\end{rmk}

\begin{proof}
	The proof follows from a quick induction: notice that
\begin{align*}
	\{
		(\ell_1, \ell_2, &\ldots, \ell_{r-1}) 
		\mid 
		\ell_1 \in t_1 \times t_2, \; 
		\ell_{j+1} \in \ell_j \times t_{j+2}, \; 
		j=1, \ldots, r-3, \; 
		b \in \ell_{r-2} \times t_r
	\}
	\\
	&= 
	\bigcup_{\ell_1 \in t_1 \times t_2}
	\{
		(\ell_1, \ell_2, \ldots, \ell_{r-1}) 
		\mid 
		\ell_{j+1} \in \ell_j \times t_{j+2}, \; 
		j=1, \ldots, r-3, \; 
		b \in \ell_{r-2} \times t_r
	\}.
	\qedhere
\end{align*}
\end{proof}

The next theorem obtains formulas for the action of braid generators on $\Hom(a^{\times m}, b)$ with respect to $\mathcal{T}_m$. We consider odd- and even-indexed generators separately. \Cref{diag braid gen} shows each odd-indexed generator acts diagonally on $\mathcal{T}_m$; our computational basis diagonalizes at least half the braid generators. The action of even-indexed generators is more complicated but it boils down to a calculation in either $\Hom(a^3, b)$ or $\Hom(a^4, b)$, depending on the parity of $m$, so we single out these important special cases in the next lemma.

\begin{lem}\label[lem]{sig2 in b3 and b4}
With respect to $\mathcal{T}_m$, the braid generator $\sigma_2 \in B_3$ acts on $\Hom(a^{\otimes 3}, b)$ as the matrix
\[
(\sigma_2)_{\ket{x}, \ket{y}} 
	= \sum_{z \in \ccat} 
	\left[ F^{aaa}_{b;\, zy} \right]^{-1} R^{aa}_y F^{aaa}_{b;\,yx}.
\]
Similarly, $\sigma_2 \in B_4$ acts on $\Hom(a^{\otimes 4}, b)$ as the matrix
\[
(\sigma_2)_{\ket{x x'}, \ket{y y'}}
	= \sum_{c, d}
	F^{aay'}_{b;\,yc} 
	\left[F^{aaa}_{c;\,y' d} \right]^{-1}
	R^{aa}_{d}
	F^{aaa}_{c;\,d x'}
	\left[F^{aa x'}_{b;\,cx} \right]^{-1}.
\]
\end{lem}
\begin{proof}
	Our proof is diagrammatic. \Cref{odd one out three strands,mid sig} compute the action of $\sigma_2$ on $\Hom(a^3, b)$ and $\Hom(a^4, b)$ with respect to $\mathcal{T}_m$ using a sequence of $F$- and $R$-moves. In both cases the sums range over all admissible trees.
\end{proof}

With \Cref{sig2 in b3 and b4} in hand, we may easily compute the action of every braid generator.

\begin{thm}\label{general braiding formula}
	For any $m$ and any $j = 1, \ldots, r$, the braid generator $\sigma_{2j-1}$ in $B_m$ acts  on $\mathcal{T}_m$ as the matrix
\begin{equation}\label{odd sig}
	\left(\sigma_{2j-1}\right)_{\ket{t \ell}, \ket{t' \ell'}} 
	= 
	R^{aa}_{t_j} \delta_{\ket{t \ell}, \ket{t' \ell'}}.
\end{equation}
	Now suppose $m = 2r + 1$ and consider $\sigma_2^{(3)} \in B_3$. Then $\sigma_{2r}$ in $B_m$ acts on $\mathcal{T}_m$ as the matrix
\[	
	\left(\sigma_{2j}\right)_{\ket{t'\ell'}, \ket{t\ell}} 
	=
	\sum_c
	F^{\ell_{r-1} t_r a}_{b;\, \ell_{r-2} c} 
	\left(\sigma_2^{(3)}\right)_{\ket{t_r'}, \,\ket{t_r}}
	\left[F^{\ell_{r-1} t_r' a}_{b; \, c f} \right]^{-1}
\]	
	Conversely, suppose that $m = 2r$, or that $m = 2r + 1$ and $j < r$, and consider $\sigma_2^{(4)} \in B_4$. Then, $\sigma_{2j}$ in $B_m$ acts on $\mathcal{T}_m$ as the matrix
\[
	\left(\sigma_{2j}\right)_{\ket{t'\ell'}, \ket{t\ell}} 
	= 
	\sum_c
	F^{\ell^* t_{j} t_{j+1}}_{b^*;\, \ell_{j-1} c} 
	\left(\sigma_2^{(4)}\right)_{\ket{t_j' t_{j+1}'}, \,\ket{t_j t_{j+1}}}
	\left[F^{\ell^* t_j t_{j+1}}_{b^*; \, c f} \right]^{-1}
\]
Here $\ell^* = \ell_{j-2}$ if $j \geq 2$, $\ell^* = t_1$ if $j = 1$, $b^* = \ell_j$ if $j < r-2$, and $b^* = b$ if $j = r-2$.
\end{thm}
\begin{proof}
	Again, our proof is diagrammatic. \Cref{sig2j} computes the action $\sigma_{2j}$ on $\Hom(a^m, b)$ using a sequence of $F$- and $R$-moves. The second equality uses \Cref{sig2 in b3 and b4} to apply $\sigma_2^{(4)}$ to a subspace of $\Hom(a^m, b)$ isomorphic to $\Hom(a^4, c)$. \Cref{odd one out} illustrates the slight modification necessary when $m = 2r+1$ and $j = r$: in this case we apply $\sigma_2^{(3)}$ to a subspace of $\Hom(a^m, b)$ isomorphic to $\Hom(a^3, c)$. 
\end{proof}

\section{Numerical computations}
\label{numerical braids}
In this section we illustrate a few explicit computations performed using the \texttt{FusionRing} class implemented in {\sc SageMath}. In particular, we explicitly identify a few traditional quantum circuit gates in the image of the braid representations induced by the so-called \textit{metaplectic} and \textit{Fibonacci} anyons \cite{hastings_nayak_wang_2013,trebst_troyer_wang_ludwig_2008}.

For instance, we may easily reproduce the calculations in Section~3.2 of \cite{cui_wang_2015} using a few lines of code. To begin, we construct a \texttt{FusionRing} modeling the anyon system $SU(2)_4$, with objects labeled by the $\mathfrak{su}(2)$ weights $j/2$ for $0 \leq 0 \leq 4$. We use the \texttt{fusion\_labels}
\[
	0 \leftrightarrow \mathbf{1},
	\qquad
	1/2 \leftrightarrow X_\epsilon,
	\qquad
	1 \leftrightarrow Y,
	\qquad
	3/2 \leftrightarrow X_\epsilon',
	\qquad
	2 \leftrightarrow Z
\]
specified in \cite[Section~3.2]{cui_wang_2015} for notational consistency.

\begin{python}
su24 = FusionRing('A1', 4)
su24.fusion_labels(['one', 'X_e', 'Y', 'X_ep', 'Z'], inject_variables=True)
\end{python}
Injecting variables allows us to verify that indeed \pyth{X_e**2 == one + Y}. 

Next we construct explicit matrices describing the representation $\rho^{X_\epsilon}_4 \colon B_4 \to \Hom(X_\epsilon, Y)$ with respect to the computational basis described in \Cref{comp basis}.
\begin{python}
comp_basis, sig = su24.get_braid_generators(X_e, Y, 4)
\end{python}
We may verify that indeed \pyth{comp_basis == [(Y, Y), (Y, one), (one, Y)]}. Section~3.2 in \cite{cui_wang_2015} considers the ordered basis $\{\ket{YY}, \ket{\mathbf{1}Y}, \ket{Y\mathbf{1}}\}$ so we re-order our basis here for consistency.
\begin{python}
T = Permutation([1, 3, 2]).to_matrix()
sig = [T * ss * T.inverse() for ss in sig]
\end{python}
Finally we may verify the matrices printed in \cite{cui_wang_2015}. For instance, 
\begin{python}
omega = su24.root_of_unity(2/3)
gamma = su24.root_of_unity(1/12)
sig[0] / gamma, sig[2] / gamma
\end{python}
results in the matrices
\begin{equation*}
\left[\begin{array}{rrr}
1 & 0 & 0 \\
0 & \zeta_{48}^{8} - 1 & 0 \\
0 & 0 & 1
\end{array}\right]
\quad\text{and}\quad
\left[\begin{array}{rrr}
1 & 0 & 0 \\
0 & 1 & 0 \\
0 & 0 & \zeta_{48}^{8} - 1
\end{array}\right].
\end{equation*}
Here $\zeta_{48} = e^{{2\pi i \over 48}}$ denotes the cyclotomic field generator \pyth{su24.field().gen()}. Of course $\texttt{omega} ==  \zeta_{48}^{8} - 1$. Similarly, we may verify Cui and Wang's construction of the Hadamard gate, up to a phase: the code
\begin{python}
p = sig[0] * sig[1] * sig[0] / gamma**3
q = sig[2] * sig[1] * sig[2] / gamma**3
H = q**2 * p * q**2
sqrtd = su24.field()(sqrt(3))
su24.root_of_unity(1/2) * sqrtd * H
\end{python}
prints
\[
\left[\begin{array}{rrr}
1 & 1 & 1 \\
1 & \zeta_{48}^{8} - 1 & -\zeta_{48}^{8} \\
1 & -\zeta_{48}^{8} & \zeta_{48}^{8} - 1
\end{array}\right].
\]
Equivalently, we could define the Hadamard gate as a matrix \texttt{H}, cast it as an element of the \texttt{MatrixGroup} generated by \texttt{sig}, and then using the \texttt{word\_problem} method to obtain an expression for \texttt{H} as a word in the elements of the list \texttt{sig}.

In addition, we may ask for the structure of the group generated by the braid generators, modulo the global phase $\gamma = e^{{2 \pi i \over 12}}$.
\begin{python}
G = MatrixGroup([ss / gamma for ss in sig])
G.structure_description()
\end{python} 
Indeed, we may verify that \pyth{G.cardinality() == 648} as claimed.

Notice we may just as easily deal with metaplectic anyons belonging to larger fusion rings. In particular, the following code computes the braid representations induced by metaplectic anyons in $SO(2r+1)_2$. The metaplectic anyon $X_\epsilon$ is labeled by the half-integral weight in $\mathfrak{so}_{2r+1}$ and $Y_1$ corresponds to the first fundamental weight $(1, 0, \ldots, 0)$. 
\begin{python}
r = 5
sor = FusionRing(['B', r], 2)
Y1 = sor([1] + [0]*(r-1))
X_e = sor([1/2]*r)
comp_basis, sig = sor.get_braid_generators(X_e, Y1, 4)
G = MatrixGroup(sig)
G.structure_description()
\end{python}
The result of this calculation explains the \textit{metaplectic anyon} terminology: the image of the induced representation is a metaplectic group \cite{goldschmidt_jones}.

Similarly, we may just as easily explore braid representations on more strands. 
\begin{python}
n_strands = 7
comp_basis, sig = sor.get_braid_generators(X_e, Y1, n_strands)
\end{python}

As another example, we consider building Pauli gates $X, Y, Z$ as words in the image of a braid representation induced by the Fibonacci anyon, which we label by $\tau$ as usual. In particular, we consider the representation $\rho^\tau_3 \colon B_3 \to \Hom(\tau^{\times 3}, \tau)$.
\begin{python}
fr = FusionRing('G2', 1)
fr.fusion_labels(('one', 'tau'), inject_variables=True)
comp_basis, sig = fib.get_braid_generators(tau, tau, 3)
\end{python}

The image of $\rho^\tau_3$ is dense in $SU(2)$, so in principle we could approximate any single-qubit gate, to arbitrary accuracy, using an appropriate word in $B_3$.  To illustrate the approximation process, we follow Section~V.B in \cite{bonesteele} and construct so-called \textit{weaves} approximating the gate $i X$. We note that in this case, {\sc SageMath} cannot simply solve the word problem in the image of $\rho^\tau_3$ because the generating matrices contain non-cyclotomic entries. 
\begin{python}
CCmat = MatrixSpace(CC, 2, 2)
target = CCmat([0, I, I, 0])
pattern = weave_searcher(target, max_len=11, tol=1e-2)
weave = sig[0].parent().one()
for j in range(len(pattern)):
    weave = sig[j
\end{python}
A naive implementation of the brute-force \texttt{weave\_searcher} is available at \url{https://github.com/willieab/weave_searcher}.

\section{$F$-matrix solver implementation}
\label{solver details}
In this section we discuss our implementation of the orthogonal $F$-matrix solver in some detail. The solver is available as the {\sc SageMath} method \texttt{FMatrix.find\_orthogonal\_solution}. Currently, the code is available on the {\sc SageMath} development branch at \url{https://trac.sagemath.org/ticket/30423}; it is set to merge into the stable {\sc SageMath} $9.8$ release. In a nutshell, the solver computes a solution to the Pentagon Equations \eqref{pent eqn} associated to a given \texttt{FusionRing} object using Groebner basis methods; essentially, the solver implements various techniques to ensure that the calculation, which typically scales exponentially with the number of variables and equations and their degrees, remains tractable \cite{buchberger_1976}. 

The main novelties are that we exploit the Hexagon Equations \eqref{hex eqn}, enforce orthogonality, and partition the system at crucial moments in the computation according to the \textit{equations graph} defined in \ref{eqn graph}. We learned that: the Hexagon equations alone determine a significant fraction of the unknowns, and they help determine a field containing the $F$-matrix as they feature $3j$-symbols in the form of cyclotomic coefficients; that including the orthogonality constraints shrinks the solution variety so it makes the search for a Groebner basis more efficient; and that partitioning allows the solver to consider relatively small subsets of equations independently and in parallel.

\begin{dfn}\label[defn]{eqn graph}
	Given a set $P$ of polynomials in $\mathbb{C}[x_1, \ldots, x_n]$, let $\mathcal{G}(P)$ denote the associated undirected \textit{equations graph} with vertices labeled by $x_1, \ldots, x_n$ and an edge $x_i \to x_j$ if and only if there is a polynomial in $P$ with a non-zero term divisible by $x_i x_j$. That is, nodes in $\mathcal{G}(P)$ correspond to variables and two nodes are connected whenever the corresponding variables appear together in an element of $P$.
\end{dfn}

The solution algorithm consists of three main steps:
\begin{enumerate}
	\item find the $F$-symbols determined by the Hexagon Equations \eqref{hex eqn} together with the orthogonality constraints
\begin{equation*}
	\big[F^{abc}_d \big]^T \big[F^{abc}_d \big] = I, 
	\quad\text{for every } (a, b, c, d) \in \ccat^4 \text{ such that } F^{abc}_d \neq 0;
\end{equation*}
	\item substitute into the Pentagon Equations \eqref{pent eqn} and eliminate variables iteratively using various reduction heuristics; and finally
	\item obtain a numerical solution by solving the few remaining relations amongst the few remaining unknowns.
\end{enumerate}
We note that requiring orthogonality instead of unitarity avoids duplicating the number of variables: unitarity necessitates complex conjugates for each unknown $F^{abc}_{d; \, ef}$. This choice seems to have no practical bearing, since all the $6j$-systems we have obtained turn out to be real, which means each matrix $F^{abc}_{d}$ is real orthogonal and therefore unitary.

Each of the three main steps involves a Groebner basis calculation that nevertheless remains tractable in many interesting cases for the following reasons. First, the graph defined by the Hexagon Equations together with the orthogonality constraints consists of many relatively small connected components that are processed independently and in parallel. Regardless of the graph structure, the solver ignores \textit{large} connected components: those with more nodes than allowed by the optional \texttt{max\_component\_size=45} parameter. In all cases studied, the solver obtains over $50\%$ of the $F$-symbols in this step, without even setting up the Pentagon Equations. 

Second, the reduction heuristics turn out to be rather powerful as they significantly decrease the number of equations and variables in the pentagon system through repeated back substitution. In any case, most pentagons vanish already because many $F$-symbols are obtained in Step $(1)$. 

Third, the equations graph of the reduced system consists of isolated points and a few small connected components, if any. In particular, typically less than $0.25\%$ of the Pentagon Equations remain at this stage, and most of them are univariate quadratics. This usually obviates the need for a Groebner basis calculation in Step $(3)$. 

The following concrete example illustrates the typical flow of the solution algorithm.  In the Type $B$ fusion ring $SO(21)_2 = \texttt{FusionRing('B10', 2)}$, there are over $1.5$ million Pentagon Equations in $17,437$ variables. Step $(1)$ determines $11,540$ $F$-symbols. At the start of Step $(2)$, the solver sets up only $651,173$  Pentagon Equations in the remaining $5,897$ variables. The elimination loop obtains an equivalent system of $193$ relations amongst $193$ unknowns. Every relation in the reduced system is of the form $x_j^2 - \alpha$, for some constant $\alpha$, so in Step $(3)$ the solver obtains a solution using simple root finding methods. We note that the fusion rings $SO(n)_2$ arise in connection to the duality results discussed in \crossrefglnch\ and \crossrefsonch.

At the end of a successful calculation, the $F$-symbols may be retrieved using the \texttt{FMatrix.get\_fvars} method. They are reported as elements of a common number field. Mathematically, this field is the compositum of the associated \texttt{FusionRing}'s \texttt{field} and the extension of $\mathbb{Q}$ defined by the product of all terms remaining in Step $(3)$. Computing a defining polynomial for this compositum turns out to be intractable sometimes, so in certain cases the $F$-symbols are returned as elements of {\sc SageMath}'s generic \texttt{AlgebraicField}. These cases were determined using experimental data collected in March $2021$.

We note that the elimination loop of Step $(2)$ typically accounts for over $80\%$ of the processing time. Each iteration consists of two steps: first find new $F$-symbols (possibly in terms of others that come later in the lexicographical order) and then update the remaining polynomials with the new expressions. Following \cite[Section~2.5]{bonderson}, a triangular solver extracts $F$-symbols from the ideal basis elements resulting from the Groebner basis calculation at the end of Step $(1)$. In fact we solve only \textit{easy} equations: those defined by a polynomial $p$ with at most two terms, one of which is univariate and linear in the largest variable that appears in $p$.
This simple reduction step turns out to be rather powerful; there is no need to solve for higher-order terms, which typically have multiple solutions and require sophisticated branching methods to keep track of all possibilities, or for linear terms in longer polynomials, which result in solutions whose repeated back-substitution quickly places a heavy burden on system memory. 

Solving easy equations is fast, so the update step in fact accounts for the vast majority of the processing time. During this update step, the solver also reduces each polynomial in the ideal basis modulo its leading coefficient, its greatest common factor known to be non-zero, and modulo every quadratic of the form $x_j^2 - \alpha$, for some constant $\alpha$. This reduction step typically produces new two-term equations so the elimination loop continues until no new easy equations are found. Even though our custom arithmetic engine provides the fast \texttt{update\_reduce} method that operates on the polynomial's sparse exponent vector directly at the {\sc C} level, updating and reducing polynomials remains a bottleneck.

To accelerate this and other lengthy calculations, our implementation relies on concurrent programming to distribute tasks amongst several worker processes.

\subsection{Parallel computations}
\label{parallel details}

The solver leverages the parent-child paradigm and our own bare-bones implementation of the \texttt{MapReduce} protocol introduced in \cite{mapreduce} to split up \textit{embarrassingly parallel} tasks: those whose dependency graph is trivial so they can be completed independently without requiring synchronization amongst workers.

{\sc Python}'s Global Interpreter Lock (GIL) makes it so that only a single thread may be active at any one time, so we use multiple processes, instead of multiple threads, to achieve true concurrency. Each process owns its own memory space, and no process may access any other's memory blocks. This means that any piece of solver state required by a worker must be pickled and communicated over pipes. In practice, however, the Inter-Processor Communication (IPC) overhead proved prohibitively slow, so the solver works hard to keep the communication costs at a minimum.

For instance, the solver avoids piping \texttt{self}, the \texttt{FMatrix} object managing the whole calculation, by leveraging the fact that forking produces identical copies of the parent's virtual memory in each child process. This means we may reference the copy of \texttt{self} in each child process using the \textit{virtual} address \texttt{id(self)} computed in the parent process. Thus we need only pass method \textit{names} and the address \texttt{id(self)} to worker processes: each process can then bind the corresponding method to its own copy \texttt{self}, located at \texttt{id(self)}.

As another example, the solver avoids passing long lists of label tuples when instructing worker processes to set up the Hexagon and Pentagon Equations; instead, it passes each of the \texttt{n\_proc} workers a unique index $0 \leq \texttt{child\_id} < \texttt{n\_proc}$ and instructs each worker to enumerate all possible tuples but to process only the ones with index \texttt{child\_id} modulo \texttt{n\_proc}.

In the spirit of reducing IPC further, we implemented various pieces of solver state using shared memory-backed data structures. The following subsection describes these in detail. 

A notable exception is the \texttt{ideal\_basis}. This {\sc Python} \texttt{list} keeps track of the relations amongst the unknowns that remain at each step of the calculation. The parent process owns the \texttt{ideal\_basis}. At each elimination round, it pipes the \texttt{ideal\_basis} to child processes in chunks for updating using the \texttt{MapReduce} protocol: the \texttt{update\_reduce} method is the mapper and the reducer simply collects all polynomials and discards duplicates. Experimental evidence suggests the optimal chunk size is $\lfloor \texttt{len(ideal\_basis)} / \texttt{n\_proc}^2 \rfloor + 1$: when there are too many chunks the IPC overhead dominates and when there are too few some workers idle. 

\subsection{Shared solver state}
\label{data structures}

We describe the state variables maintained by our solver, as summarized in \Cref{state vars}. Most crucial data structures reside in shared memory because the solver performs parallel computations and IPC is rather costly in practice.

\begin{table}[!h]
\centering
\begin{tabular}{c|c|c}
Name & Class &  Updated by \\
\hline
\texttt{fvars} & \texttt{FvarsHandler} & Parent \\
\texttt{ks} & \texttt{KSHandler} & Parent \\
\texttt{solved} & \texttt{list} & Parent \\
\texttt{var\_degs} & \texttt{list} & Parent \\
\end{tabular}
\caption{Shared memory-backed variables maintained by the orthogonal $F$-matrix solver.}
\label{state vars}
\end{table}

First we describe our polynomial representation. We treat each $6j$-symbol $F^{abc}_{d; \, ef}$ as an unknown and assign a unique linear index to each \textit{admissible} sextuple $(a, b, c, d, e, f)$. The \texttt{idx\_to\_sextuple} dictionary manages this map. This means we view every polynomial equation as an identity of the form $f(x) = 0$, for some $f \in \cycfield[x_1, \ldots, x_n]$, with $\cycfield$ denoting the associated \texttt{FusionRing}'s cyclotomic \texttt{field} and with $n = \texttt{len(idx\_to\_sextuple)}$. The polynomial ring $\mathcal{P} \coloneqq \cycfield[x_1, \ldots, x_n]$ plays an important role in this section. Unless otherwise stated, we consider $\mathcal{P}$ equipped with the partial order induced by the \textit{degree reverse lexicographical (degrevlex)} monomial ordering.

Internally, the solver represents each element of $\mathcal{P}$ as a degrevlex-ordered tuple of exponent-coefficient pairs. In turn, each exponent is implemented as a {\sc SageMath} \texttt{ETuple} while its corresponding cyclotomic coefficient as a {\sc Python} tuple. We use \texttt{ETuple}s for exponents because they can efficiently handle sparse vectors: the monomials we encounter tend to be very sparse and an \texttt{ETuple} stores the sparse exponent vector $(p_1, \ldots, p_n)$ on $x_1^{p_1} \cdots x_n^{p_n}$ compactly as the {\sc C} \texttt{int} array $(i_1, p_{i_1}, \ldots, i_k, p_{i_k})$, with $i_j$ denoting the subset of indices with $p_{i_j} \neq 0$. Conversely, we use a fixed-length dense representation for the coefficients for speed considerations: we store each cyclotomic coefficient as a tuple of $d$ {\sc SageMath} \texttt{Rational}s, with $d$ denoting the degree of $\cycfield$. Thus, for instance, if $\cycfield = \mathbb{Q}[\zeta_6]$ with $\zeta_6 = e^{i \pi / 3}$ and $n = 25$, we implement the polynomial
\begin{align}
	p(x) &= {2^{65} \over 2^{127} + 1} x_1^3 x_9 x_{13}^2 - \left({19 \over 7} \zeta_6 - 2^{129}\right) x_{21}^7
	\in \mathcal{P}
	\label{polynomial example}
	\intertext{as the {\sc Python} tuple}
	\Big(
		\big(
			(1, 3, 9, 1, &13, 2), \, 
			\left(2^{65} / (2^{127} + 1), \, 0 \right)
		\big), 
	\,\, 
		\big(
			\left(21, 7\right), \, \left(2^{129}, \, -19/7\right)
		\big)
	\Big).
	\label{internal poly repn}
\end{align}

Although {\sc SageMath} offers various implementations for multivariate polynomials, we developed our own for two main reasons. First, controlling the internal implementation allowed for the development of a custom arithmetic engine in {\sc Cython} that exploits sparsity and provides fast implementations for several core manipulations. In a large calculation, the solver calls the arithmetic engine methods millions of times, so it is important that these methods are fast and that they exploit fast {\sc Cython} dispatch protocols. Second, our internal implementation circumvents an issue we encountered when {\sc SageMath}'s multi-threaded \texttt{CyPari2} attempts to pipe the built-in polynomial classes from single-threaded child processes back to the main parent process.

The \texttt{fvars} structure records the current state of each unknown. The associated \texttt{solved} attribute maintains a list of $n$ booleans indicating which $F$-symbols are known, potentially as a polynomial in variables that are smaller with respect to the degrevlex ordering. Initially, \texttt{fvars} implements a dictionary mapping an admissible sextuple to a corresponding generator of $\mathcal{P}$. As the calculation progresses and the solver determines certain $F$-symbols in terms of smaller ones, the parent process performs back-substitution to update entries in \texttt{fvars} and \texttt{solved}. Thus, in general, \texttt{fvars} maintains a polynomial in $\mathcal{P}$ for each admissible sextuple.

The mapping represented by \texttt{fvars} could easily be implemented using a {\sc Python} dictionary, but there is no good means of sharing such complex {\sc Python} objects amongst processes without incurring significant IPC overhead; for this reason, we implemented the \texttt{FvarsHandler} class. This class emulates a {\sc Python} dictionary syntactically, supporting special {\sc Python} methods like \texttt{\_\_getitem\_\_} and \texttt{\_\_setitem\_\_} using the familiar bracket assignment syntax, while allowing several processes to read from the same piece of contiguous {\sc C}-level  memory simultaneously. 

The \texttt{FvarsHandler} class is implemented as a raw {\sc C}-style shared memory block provided by the \texttt{multiprocessing.shared\_memory} module. The shared memory block must be pre-allocated and later populated by entries of a fixed data type expressible as a collection of {\sc C} types. However, it can be accessed via \texttt{NumPy}'s buffer interface. 

Thus we view the shared memory block as a \texttt{NumPy} record array storing elements of the structured data type \texttt{fvars\_t}, which decomposes the internal polynomial representation illustrated in \eqref{internal poly repn} into collections of \texttt{NumPy} integers as follows. We consider the monomials and the cyclotomic coefficients separately. On one hand, we store the exponent data contiguously in a single one-dimensional array under the \texttt{exp\_data} field. The associated \texttt{ticks} field manages an array indicating the number of non-zero exponents in each term: this array indicates when the data in \texttt{exp\_data} ``jumps'' from term to term. The \texttt{ticks} array has length \texttt{max\_terms}, a parameter with default value $20$, which specifies the maximum number of terms in any representable polynomial. The length of \texttt{exp\_data} is \texttt{k*max\_terms}, for some fixed \texttt{k}. By default, \texttt{k=30}. For example, if $\texttt{max\_terms} = 4$ and $\texttt{k} = 3$, the \texttt{exp\_data} and \texttt{ticks} arrays representing the polynomial $p$ of \eqref{internal poly repn} are given by
\begin{align*}
\texttt{exp\_data} &= 
\begin{bmatrix} 1 & 3 & 9 & 1 & 13 & 2 & 21 & 7 & 0 & 0 & 0 & 0\end{bmatrix}
\\
\texttt{ticks} &= \begin{bmatrix} 3 & 1 & 0 & 0\end{bmatrix}.
\end{align*}

On the other hand, we store the cyclotomic coefficient data using two three-dimensional arrays: the \texttt{coeff\_nums} and \texttt{coeff\_denoms} arrays store $d$ numerators and denominators for each term. It would seem that two-dimensional arrays of shape \texttt{(max\_terms, d)} would suffice; however, native {\sc Python} integers can be arbitrarily large so to avoid overflowing $64$-bit \texttt{NumPy} integers we store the \textit{digits} of each signed numerator in base $2^{63}$ and of each \textit{un}signed denominator in base $2^{64}$. The \texttt{n\_bytes} parameter, with default value $32$, specifies the number of bytes pre-allocated to the numerator and denominator of each rational coefficient in \eqref{internal poly repn}. Thus each cyclotomic coefficient becomes an array with shape \texttt{(d, n\_bytes//8)}. For instance, if $\texttt{n\_bytes} = 24$, we store the coefficient data of the polynomial $p$ in \eqref{polynomial example} as
\begin{align*}
	\texttt{coeff\_nums[0, :, :]} &=
	\begin{bmatrix}
		0 & 4 & 0\\
		0 & 0 & 0
	\end{bmatrix} 
	&
	\texttt{coeff\_nums[1, :, :]} &=
	\begin{bmatrix}
		0 & 0 & 8\\
		-19 & 0 & 0
	\end{bmatrix} \\
	\texttt{coeff\_denoms[0, :, :]} &=
	\begin{bmatrix}
		1 & 2^{63} & 0\\
		1 & 0 & 0
	\end{bmatrix} 
	&
	\texttt{coeff\_denoms[1, :, :]} &=
	\begin{bmatrix}
		1 & 0 & 0\\
		7 & 0 & 0
	\end{bmatrix}
\end{align*}

Since we store only a collection of integers, retrieving a record involves costly {\sc Python} object instantiation. To mitigate this cost, the \texttt{FvarsHandler} implements a caching mechanism that reduces the number of times the solver constructs polynomial objects from the data stored in shared memory. Each processor  must build its own object cache, so \texttt{fvars\_t} contains a \texttt{modified} field with shape \texttt{(n, n\_proc)} indicating to each processor which entries were modified by the parent process since they were last retrieved.

Thus, all things considered, the data type \texttt{fvars\_t} contains the following fields:
\begin{itemize}
	\item \texttt{ticks}, with shape \texttt{(max\_terms,)}, stores unsigned $8$-bit integers indicating the number of non-zero exponents on each monomial;
	\item \texttt{exp\_data}, with shape \texttt{(30*max\_terms,)}, stores $16$-bit integers representing the non-zero exponents for all monomials in the corresponding polynomial;
	\item \texttt{coeff\_nums}, with shape \texttt{(max\_terms, d, n\_bytes//8)}, stores $64$-bit integers representing all numerators of the cyclotomic coefficients in base $2^{63}$;
	\item \texttt{coeff\_denoms}, with shape \texttt{(max\_terms, d, n\_bytes//8)}, stores unsigned $64$-bit integers representing all denominators of the cyclotomic coefficients in base $2^{64}$; and
	\item \texttt{modified}, with shape \texttt{(n\_proc,)}, stores $8$-bit integers indicating to each child process which entries have been modified by the parent process.
\end{itemize}

The default values for the parameters and data types for each field within the structured array were determined using experimental observations. For context, we note {\sc SageMath}'s \texttt{PolynomialRing} class can model polynomial rings with at most $2^{15}$ generators. 

Similar to the \texttt{FvarsHandler}, the \texttt{KSHandler} implements a shared mapping that manages $F$-symbols with known squares: those determined up to a sign by a quadratic of the form $x_j^2 - \alpha$, for some $\alpha \in \cycfield$. The known squares play an important role in the elimination loop, where they allow for substantial simplification when updating ideal basis elements. Like the \texttt{FvarsHandler}, the \texttt{KSHandler} also emulates a {\sc Python} dictionary syntactically and it serves multiple processes simultaneously. Moreover, the \texttt{KSHandler} uses the same memory layout as the \texttt{FvarsHandler} except but its structured data type only contains the \texttt{coeff\_nums} and \texttt{coeff\_denoms} fields.

Both the \texttt{FvarsHandler} and the \texttt{KSHandler} class are implemented in {\sc Cython} for performance reasons: they use typed {\sc Cython} \texttt{memoryview}s to avoid {\sc Python} overhead in manipulating indices and accessing buffer entries so they enjoy access to the shared memory blocks directly at the {\sc C} level. In addition, the \texttt{KSHandler} class offers an API consisting entirely of \texttt{cdef} methods. Thus, while only accessible to {\sc Cython} code,  \texttt{KSHandler} methods do not incur any {\sc Python} calling overhead.

Finally, we note that shared \texttt{var\_degs} attribute maintains a \texttt{list} recording the highest power of each known variable appearing in the \texttt{ideal\_basis} at any given time. Using it, each child process can pre-compute all the necessary powers of each $F$-symbol expression at the start of an elimination step; this significantly reduces the total arithmetic performed.

\section{Appendix: diagrammatic proof of braid generator formulas}
This appendix supplies a  diagrammatic proof of \Cref{general braiding formula}.

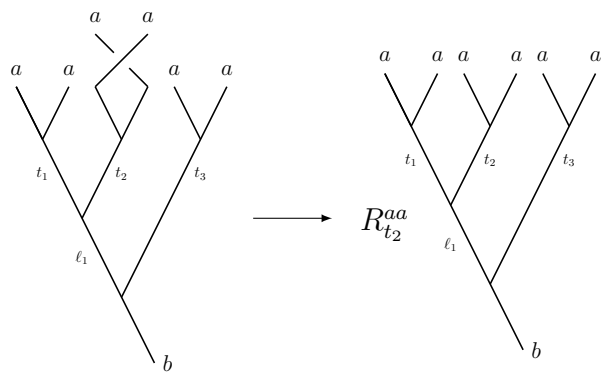
\begin{figure}[!h]
\begin{center}
\begin{tikzpicture}[scale=0.35]
	\braidedtree{6}{3}{{0,5}}
	\draw[-latex] (9, 0) -- (12,0);
	\node[scale=2] at (14, -0.25) {$R^{aa}_{t_2}$};		
	\fusiontree{6}{{14,5.5}}
\end{tikzpicture}
\end{center}	
\caption{The odd-indexed braid generators $\sigma_{2j-1} \in B_m$ act diagonally with respect to our computational basis $\mathcal{T}_m$. Here $j = 2$ and $m = 6$.}
\label{diag braid gen}
\end{figure} 

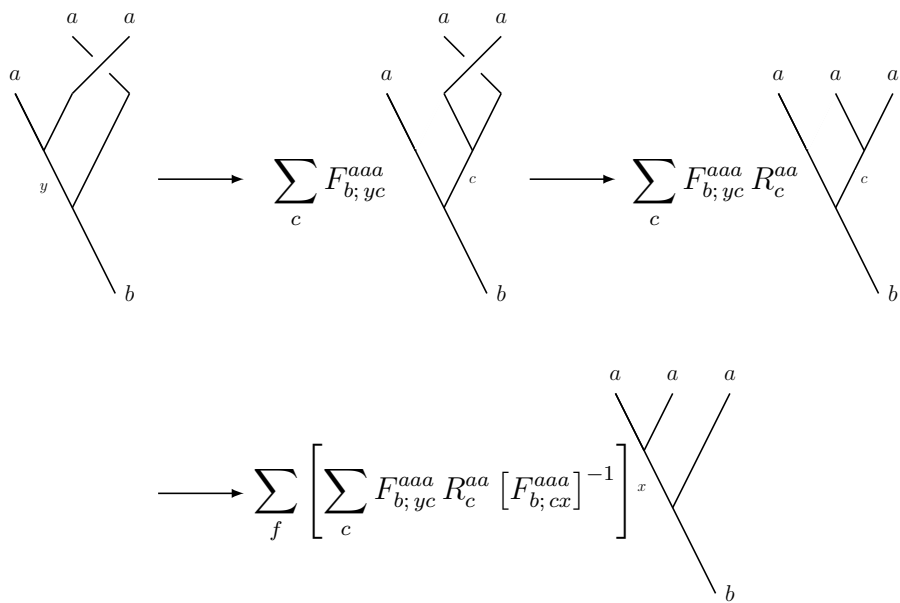
\begin{figure}[!h]
\begin{center}
\begin{tikzpicture}[scale=0.38]
	\braidedtree{3}{2}{{0,3}}
	\innerlab{{y}}
	
	\draw[semithick, -latex] (5, 0) -- (8,0);
	\node[scale=2] at (11, -0.45) {
		$
		\displaystyle \sum_c
		F^{aaa}_{b;\,yc}
		$
	};
	
	\coordinate (tlc) at (13, 3);
	\braidedtree{3}{2}{tlc}
	\innerlab{{\relax}}
	\coordinate (p) at ($(3, -3) + (tlc)$);
	\node at (p) [replace] {$c$};
	\filldraw[very thick, white] ($(2,0) + (tlc)$) -- ++($\ss*(mop)$);
	\draw[semithick] ($(2,0) + (tlc)$) -- ++(m);
	
	\draw[semithick, -latex] (18, 0) -- (21,0);
	\node[scale=2] at (24.5, -0.45) {
		$
		\displaystyle \sum_c
		F^{aaa}_{b;\,yc} \, R^{aa}_c
		$
	};	
	
	\coordinate (tlc) at (26.7, 3);
	\fusiontree{3}{tlc}
	\innerlab{{\relax}}
	\coordinate (p) at ($(3, -3) + (tlc)$);
	\node at (p) [replace] {$c$};
	\filldraw[very thick, white] ($(2,0) + (tlc)$) -- ++($\ss*(mop)$);
	\draw[semithick] ($(2,0) + (tlc)$) -- ++(m);
		
	\draw[semithick, -latex] (5, -11) -- (8,-11);
	\node[scale=2] at (15, -11) {
		$
		\displaystyle \sum_f \left[ \sum_c 
		F^{aaa}_{b;\,yc} \, R^{aa}_c \left[F^{aaa}_{b;\,cx}\right]^{-1} 
		\right]
		$
	};	
	
	\fusiontree{3}{{21,-7.5}}
	\innerlab{{x\;}}
\end{tikzpicture}
\end{center}
\caption{The action of $\sigma_2 \in B_3$ on $\Hom(a^3, b)$ with respect to our fusion tree basis.}
\label{odd one out three strands}
\end{figure}

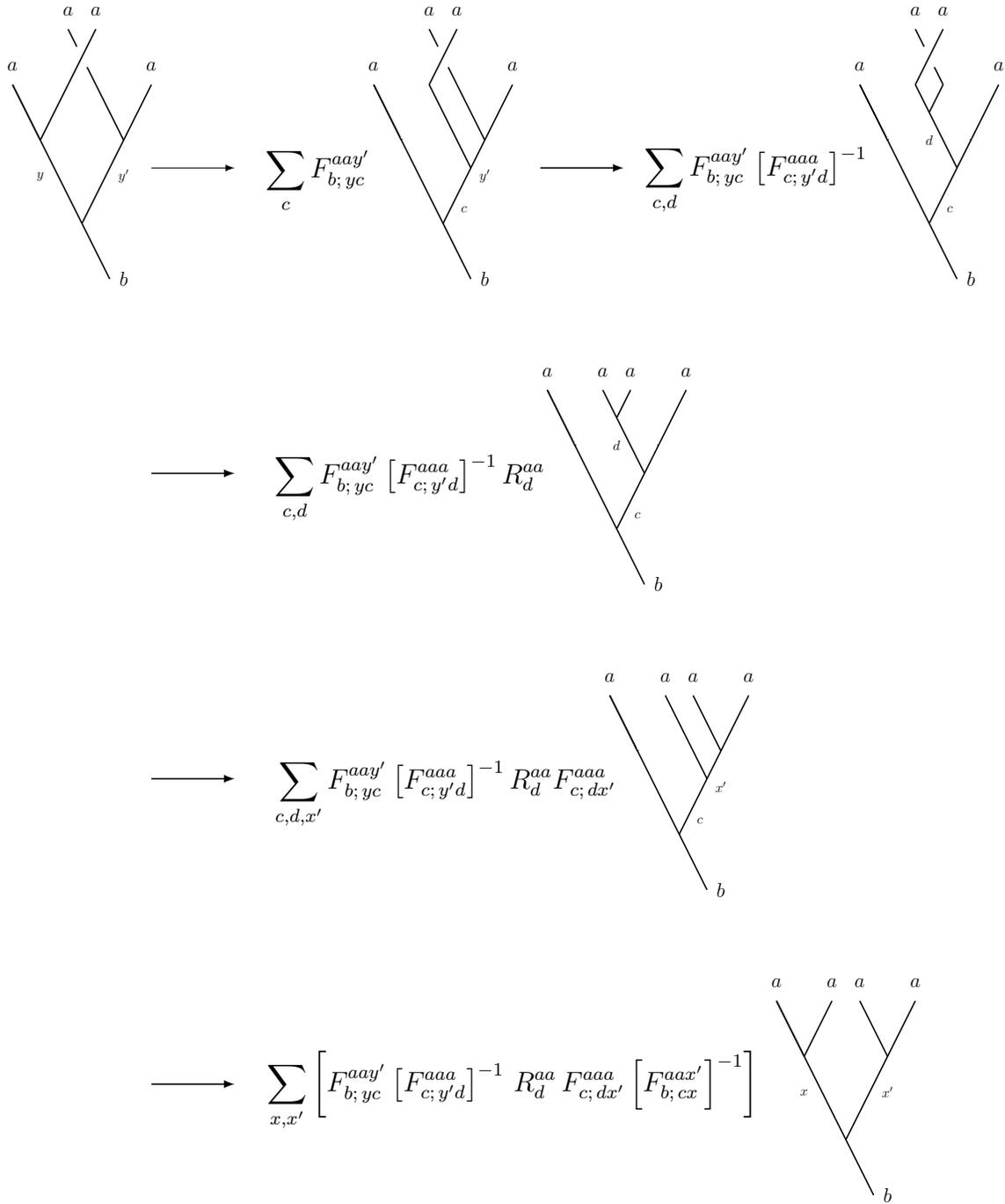
\begin{figure}[!h]
\begin{center}
\begin{tikzpicture}[scale=0.42]
	\braidedtree{4}{2}{{0,3}}
	\innerlab{{y,y'}}
		
	\draw[-latex] (5, 0) -- (8,0);
	\node[scale=2] at (11, -0.45) {
		$
		\displaystyle \sum_c
		F^{aay'}_{b;\,yc}
		$
	};
	
	\coordinate (tlc) at (13, 3);
	\braidedtree{4}{2}{tlc}
	\innerlab{{\relax, y'}}
	\coordinate (p) at ($(3.25,-4.5) + (tlc)$);
	\node at (p) [replace] {$c$};
	\filldraw[very thick, white] ($(2,0) + (tlc)$) -- ++($\ss*(mop)$);
	\draw[semithick] ($(2,0) + (tlc)$) -- ++($1.5*(m)$);	
	
	\draw[semithick, -latex] (19, 0) -- (22,0);
	\node[scale=2] at (26.75, -0.45) {
		$
		\displaystyle \sum_{c,d}
		F^{aay'}_{b;\,yc} 
		\left[F^{aaa}_{c;\,y' d} \right]^{-1}
		$
	};	
	
	\coordinate (tlc) at (30.5, 3);
	\braidedtree{4}{2}{tlc}
	\innerlab{{\relax, \relax}}
	\coordinate (p) at ($(3.25,-4.5) + (tlc)$);
	\node at (p) [replace] {$c$};
	\coordinate (p) at ($(2.5,-2) + (tlc)$);
	\node at (p) [replace] {$d$};
	\filldraw[very thick, white] ($(2,0) + (tlc)$) -- ++($\ss*(mop)$);
	\draw[semithick] ($(2,0) + (tlc)$) -- ++($1.5*(m)$);	
	\filldraw[very thick, white] ($(3,0) + (tlc)$) -- ++($\ss*(m)$);
	\draw[semithick] ($(3,0) + (tlc)$) -- ++(-0.5,-1);	
		
	\tikzmath{\yy = -11;}
	\draw[semithick, -latex] (5, \yy) -- (8,\yy);
	\node[scale=2] at (14.3, \yy -.55) {
		$
		\displaystyle \sum_{c, d}
		F^{aay'}_{b;\,yc} 
		\left[F^{aaa}_{c;\,y' d} \right]^{-1}
		R^{aa}_{d}
		$
	};	
	
	\coordinate (tlc) at (19.25, 3 + \yy);
	\fusiontree{4}{tlc}
	\innerlab{{\relax, \relax}}
	\coordinate (p) at ($(3.25,-4.5) + (tlc)$);
	\node at (p) [replace] {$c$};
	\coordinate (p) at ($(2.5,-2) + (tlc)$);
	\node at (p) [replace] {$d$};
	\filldraw[very thick, white] ($(2,0) + (tlc)$) -- ++($\ss*(mop)$);
	\draw[semithick] ($(2,0) + (tlc)$) -- ++($1.5*(m)$);	
	\filldraw[very thick, white] ($(3,0) + (tlc)$) -- ++($\ss*(m)$);
	\draw[semithick] ($(3,0) + (tlc)$) -- ++($0.5*(mop)$);	
	
	\tikzmath{\yy = 2*(-11);}
	\draw[semithick, -latex] (5, \yy) -- (8,\yy);
	\node[scale=2] at (15.6, \yy -.55) {
		$
		\displaystyle \sum_{c, d, x'}
		F^{aay'}_{b;\,yc} 
		\left[F^{aaa}_{c;\,y' d} \right]^{-1}
		R^{aa}_{d}
		F^{aaa}_{c;\,d x'}
		$
	};	
	
	\coordinate (tlc) at (21.5, 3 + \yy);
	\fusiontree{4}{tlc}
	\innerlab{{\relax, x'}}
	\node at ($(3.25,-4.5) + (tlc)$) [replace] {$c$};
	\filldraw[very thick, white] ($(2,0) + (tlc)$) -- ++($\ss*(mop)$);
	\draw[semithick] ($(2,0) + (tlc)$) -- ++($1.5*(m)$);	
	
	\tikzmath{\yy = 3*(-11);}
	\draw[semithick, -latex] (5, \yy) -- (8,\yy);
	\node[scale=2] at (18, -0.25 + \yy) {
		$
		\displaystyle \sum_{x, x'} \left[ 
		F^{aay'}_{b;\, yc} 
		\left[F^{aaa}_{c;\, y' d} \right]^{-1} \,
		R^{aa}_{d} \,
		F^{aaa}_{c;\, d x'}
		\left[F^{aa x'}_{b;\, cx} \right]^{-1}
		\right]
		$
	};	
	
	\fusiontree{4}{{27.5, 3+\yy}}
	\innerlab{{x,x'}}
\end{tikzpicture}
\caption{The action of $\sigma_2 \in B_4$ on $\Hom(a^4, b)$ with respect to our computational basis.}
\label{mid sig}
\end{center}	
\end{figure}

\begin{figure}
\begin{center}
\begin{tikzpicture}[scale=0.375]
	\braidedtree{6}{4}{{0,5}}
	\innerlab{{\ell^*, t_j, ,}}
	\node at ($(7.35, 1.75)$) [replace] {$t_{j+1}$};
	\node at ($(2.3, -1.75)$) [replace] {$\ell_{j-1}$};
			
	\draw[-latex] (9, 0) -- (12,0);
	\node[scale=1.5] at (15, -0.25) {
		$
		\displaystyle \sum_c
		F^{\ell^* t_{j} t_{j+1}}_{b^*;\, \ell_{j-1} c} 
		$
	};
	
	\coordinate (tlc) at (16.5, 5);
	\braidedtree{6}{4}{tlc}
	\innerlab{{\ell^*, \relax, \relax, \relax}}
	\filldraw[very thick, white] ($(4,-2) + (tlc)$) -- ++($\ss*1.5*(mop)$);
	\coordinate (tlcb) at ($(3,0) + (tlc)$);
	\braidedtree[blue!60]{4}{2}{tlcb}
	\filldraw[very thick, white] ($(2.5,-5) + (tlc)$) -- ++($\ss*(m)$);
	\draw[color=blue!60] ($(2.5,-5) + (tlc)$) -- ++($1.5*(mop)$);
	\node at (rootlab) [replace] {\relax};
	\innerlab{{t_j,,}}
	\node at ($(2.5,-6.5) + (tlc)$) [replace] {$c$};
	\node at ($(4.3,-3.3) + (tlc)$) [replace] {$t_{j+1}$};
	\node at ($(3.3,-7) + (tlc)$) [replace] {$ $};

	\tikzmath{\yy = -14;}
	\draw[-latex] (9, \yy) -- (12,\yy);
	\node[scale=1.5] at (18.5, -0.45+\yy) {
		$
		\displaystyle \sum_{c, t_j', t_{j+1}'}
		F^{\ell^* t_{j} t_{j+1}}_{b^*;\, \ell_{j-1} c} 
		\begingroup \color{blue!60} 
			\left(\sigma_2^{(4)}\right)_{\ket{t_j' t_{j+1}'}, \,\ket{t_j t_{j+1}}}
		\endgroup
		$
	};
	
	\coordinate (tlc) at (15 + 8, 5 + \yy);
	\fusiontree{6}{tlc}
	\innerlab{{\ell^*, \relax, \relax, \relax}}
	\filldraw[very thick, white] ($(4,-2) + (tlc)$) -- ++($\ss*1.5*(mop)$);
	\coordinate (tlcb) at ($(3,0) + (tlc)$);
	\braidedtree[blue!60]{4}{2}{tlcb}
	\filldraw[very thick, white] ($(2.5,-5) + (tlc)$) -- ++($\ss*(m)$);
	\draw[color=blue!60] ($(2.5,-5) + (tlc)$) -- ++($1.5*(mop)$);
	\node at (rootlab) [replace] {\relax};
	\innerlab{{t_j', ,}}
	\node at ($(2.75,-6) + (tlc)$) [replace] {$c$};
	\node at ($(4.25,-3.35) + (tlc)$) [replace] {$t_{j+1}'$};
	\node at ($(3.3,-7) + (tlc)$) [replace] {$ $};
	
	\tikzmath{\yy = 2*(-14);}
	\draw[-latex] (9, \yy) -- (12,\yy);
	\node[scale=1.5] at (22, -0.25 + \yy) {
		$
		\displaystyle \sum_{t_j', t_{j+1}', \ell_{j-1}'} \left[
		\sum_c
		F^{\ell^* t_{j} t_{j+1}}_{b^*;\, \ell_{j-1} c} 
		\left(\sigma_2^{(4)}\right)_{\ket{t_j' t_{j+1}'}, \,\ket{t_j t_{j+1}}}
		\left[F^{\ell^* t_j t_{j+1}}_{b^*; \, c f} \right]^{-1}
		\right]
		$
	};
	
	\fusiontree{6}{{31.5, 5+\yy}}
	\innerlab{{\ell^*, t_j', ,}}
	\node at ($(7.5,-3.3) + (tlc)$) [replace] {$t_{j+1}'$};
	\node at ($(2.4, -6.75) + (tlc)$) [replace] {$\ell_{j-1}'$};
\end{tikzpicture}
\end{center}
\caption{The action of an even-indexed braid group generator $\sigma_{2j}$ on $\Hom(a^m, b)$ with respect to our computational basis.}
\label{sig2j}
\end{figure}
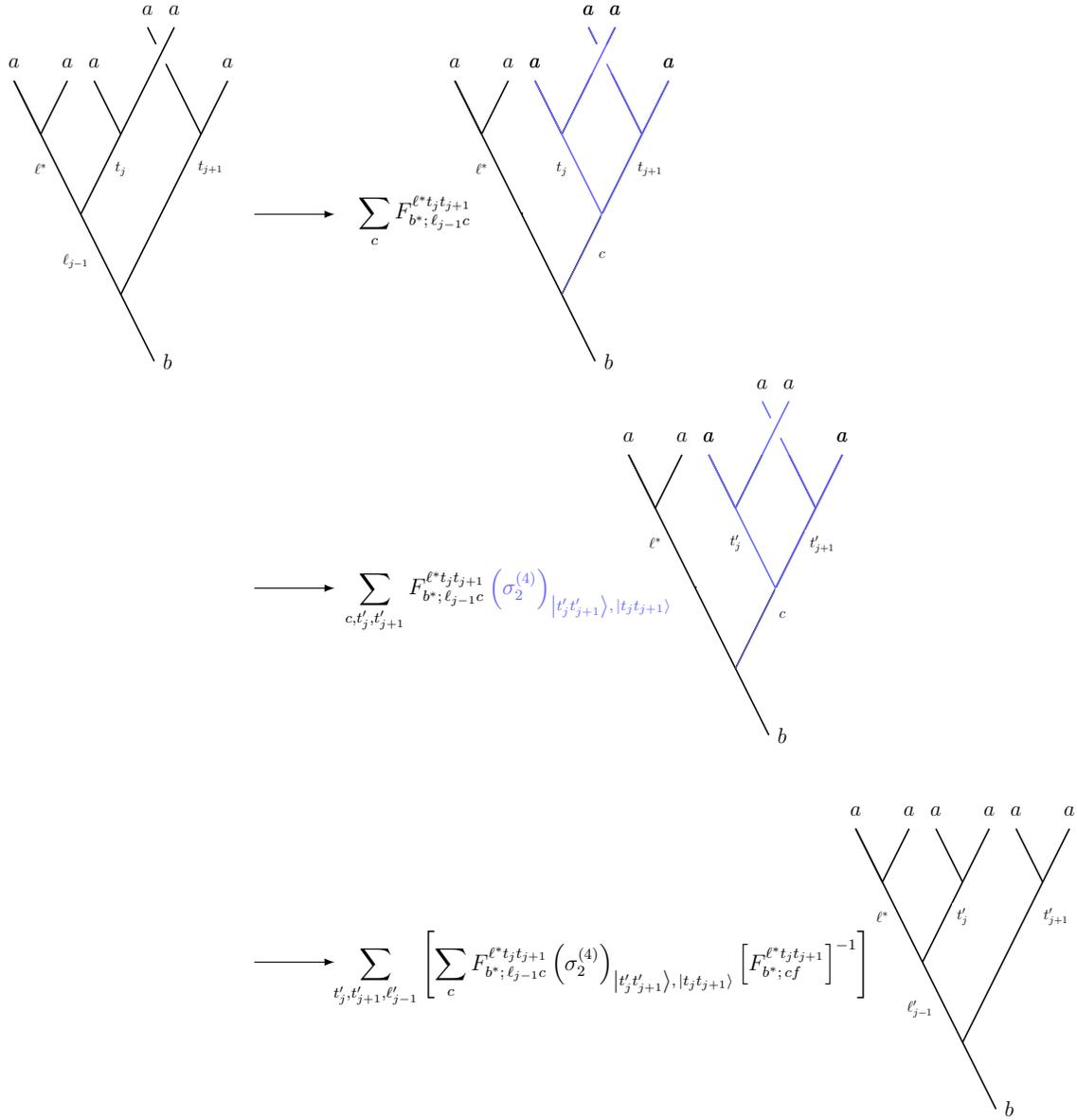

\begin{figure}
\begin{center}
\begin{tikzpicture}[scale=0.42]
	\braidedtree{5}{4}{{0,5}}
	\innerlab{{, t_r, }}
	\node at ($(.75,-3.3) + (tlc)$) [replace] {$\ell_{r-1}$};
	\node at ($(2.3,-6.3) + (tlc)$) [replace] {$\ell_{r-2}$};
		
	\draw[-latex] (8, 0) -- (11,0);
	\node[scale=2] at (14.5, -0.45) {
		$
		\displaystyle \sum_c
		F^{t_r a}_{b;\, \ell_{r-2} c} 
		$
	};
	
	\coordinate (tlc) at (17, 5);
	\braidedtree{5}{4}{tlc}
	\innerlab{{, \relax, \relax}}
	\filldraw[very thick, white] ($(4,-2) + (tlc)$) -- ++($1.5*\ss*(mop)$);
	\coordinate (tlcb) at ($(3,0) + (tlc)$);
	\braidedtree[blue!60]{3}{2}{tlcb}
	\innerlab{{t_j}}
	\node at (rootlab) [replace] {\relax};
	\filldraw[very thick, white] ($(2,-4) + (tlc)$) -- ++($1.5*\ss*(m)$);
	\filldraw[color=blue!60] ($(2,-4) + (tlc)$) -- ++($1.5*(mop)$);
	\node at ($(2,-5.5) + (tlc)$) [replace] {$c$};
	\node at ($(-2.3,-3.3) + (tlc)$) [replace] {$\ell_{r-1}$};	
	\node at ($(3.5,-7) + (tlc)$) [replace] {$ $};
	
	\tikzmath{\yy = -14;}
	\draw[-latex] (8, \yy) -- (11,\yy);
	\node[scale=2] at (17, -0.5 +\yy) {
		$
		\displaystyle \sum_{c, t_r'}
		F^{\ell_{r-1} t_r a}_{b;\, \ell_{r-2} c} 
		\begingroup \color{blue!60} 
			\left(\sigma_2^{(3)}\right)_{\ket{t_r'}, \,\ket{t_r}}
		\endgroup
		$
	};
	
	\coordinate (tlc) at (21, 5 + \yy);
	\fusiontree{5}{tlc}
	\innerlab{{, \relax, \relax}}
	\filldraw[very thick, white] ($(4,-2) + (tlc)$) -- ++($1.5*\ss*(mop)$);
	\coordinate (tlcb) at ($(3,0) + (tlc)$);
	\braidedtree[blue!60]{3}{2}{tlcb}
	\innerlab{{t_j'}}
	\node at (rootlab) [replace] {\relax};
	\filldraw[very thick, white] ($(2,-4) + (tlc)$) -- ++($1.5*\ss*(m)$);
	\filldraw[color=blue!60] ($(2,-4) + (tlc)$) -- ++($1.5*(mop)$);
	\coordinate (p) at ($(2,-5.5) + (tlc)$);
	\node at (p) [replace] {$c$};
    \node at ($(-2.3,-3.3) + (tlc)$) [replace] {$\ell_{r-1}$};
    \node at ($(3.5,-7) + (tlc)$) [replace] {$ $};
    
	\tikzmath{\yy = 2*(-14);}
	\draw[-latex] (8, \yy) -- (11,\yy);
	\node[scale=2] at (21, -0.25 + \yy) {
		$
		\displaystyle \sum_{t_r', \ell_{r-2}'} \left[
		\sum_c
		F^{\ell_{r-1} t_r a}_{b;\, \ell_{r-2} c} 
		\left(\sigma_2^{(3)}\right)_{\ket{t_r'}, \,\ket{t_r}}
		\left[F^{\ell_{r-1} t_r' a}_{b; \, c f} \right]^{-1}
		\right]
		$
	};
	
	\coordinate (tlc) at (30.25, 5 + \yy);
	\fusiontree{5}{tlc}
	\innerlab{{, t_r',}}
	\node at ($(.8,-3.3) + (tlc)$) [replace] {$\ell_{r-1}$};
    \node at ($(2.3,-6.3) + (tlc)$) [replace] {$\ell_{r-2}'$};
\end{tikzpicture}
\end{center}
\caption{The action of $\sigma_{2r}$ in $B_{2r+1}$ on $\Hom(a^m, b)$ when $m = 2r+1$.}
\label{odd one out}
\end{figure}
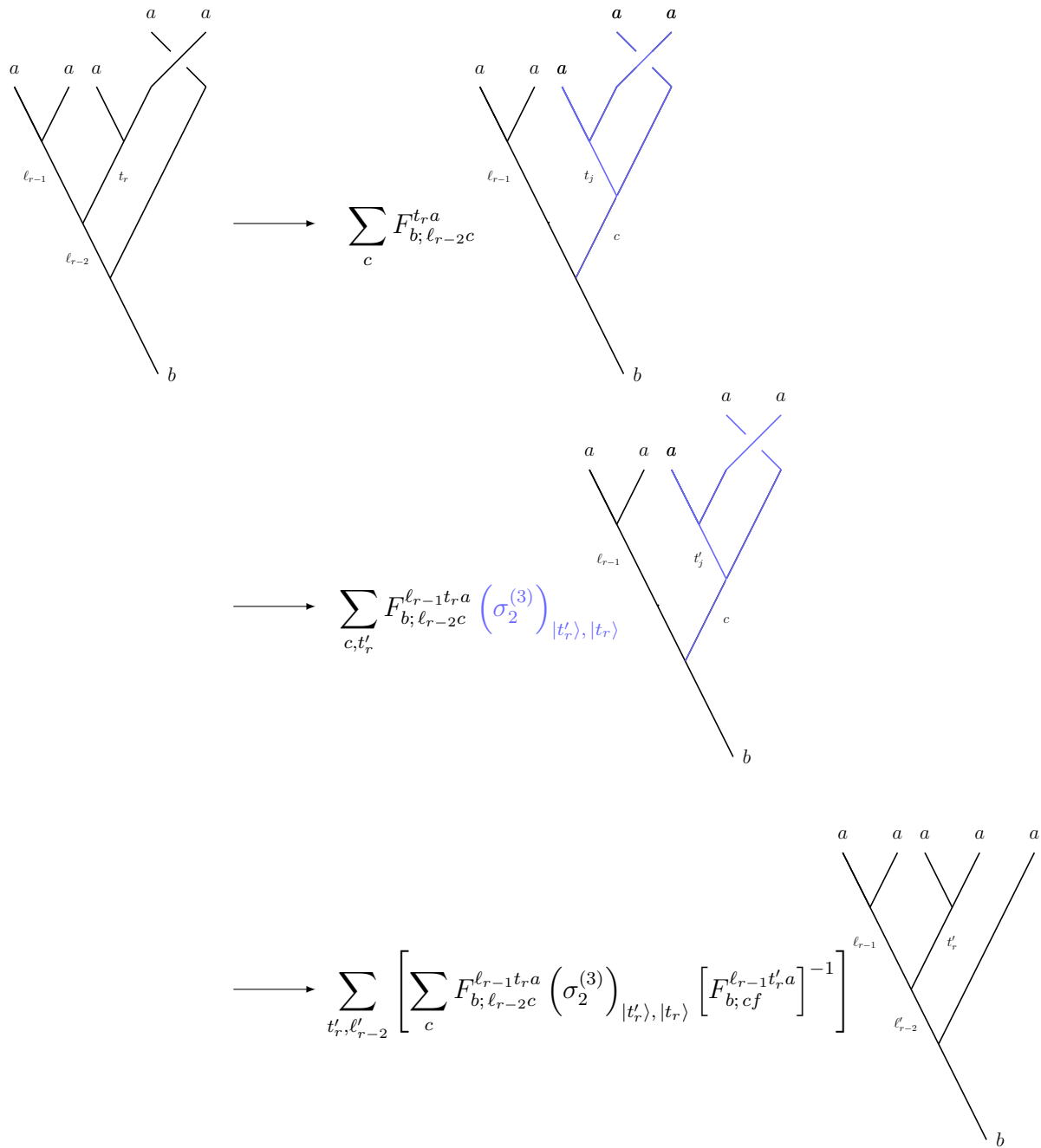


\clearpage
\bibliographystyle{alphaurl}
\bibliography{ref}
\end{document}